\documentclass[12pt]{article}
\usepackage{a4wide}

\usepackage{xcolor}
\usepackage[T1]{fontenc}
\usepackage{hyperref}
\usepackage[english]{babel} 
\usepackage{lipsum}
\usepackage{amsthm}
\usepackage{amsmath}
\usepackage{amsfonts}
\usepackage{amssymb}
\usepackage{enumerate}
\usepackage{amscd}
\usepackage{graphicx}
\usepackage{makeidx}
\usepackage[all]{xy}
\usepackage{tikz}
 \usetikzlibrary{matrix}
 \usetikzlibrary{arrows}

\usepackage[english]{babel}
\usepackage[font=small,format=plain,labelfont=bf,up]{caption}
\usepackage[bbgreekl]{mathbbol}

\usepackage{hyperref}       
\hypersetup{       
   pdftex,
   colorlinks=true,        
   pdfstartview=FitH,      
   allcolors=[rgb]{0,0.3,0.6},        
   bookmarks=true,          
	 bookmarksdepth=section, 
   bookmarksopen=false,     
   bookmarksnumbered=true, 
	 backref, pagebackref   
}

\makeindex

\newcounter{Def}[section]

\theoremstyle{plain}
\newtheorem{Theorem}[Def]{Theorem}

\newtheorem{Example}[Def]{Example}

\newtheorem{Remark}[Def]{Remark}
\newtheorem{Lemma}[Def]{Lemma}
\newtheorem{Proposition}[Def] {Proposition}
\newtheorem{Cor}[Def] {Corollary}

\theoremstyle{definition}
\newtheorem{Definition}[Def]{Definition}

\newcommand{\ipic}[3][-0.5]{\raisebox{#1\height}{\scalebox{#3}{\includegraphics{pic/#2.pdf}}}}

\newcommand{\qp}{\quad .}
\newcommand{\qc}{\quad ,}

\newcommand\blfootnote[1]{%
	\begingroup
	\renewcommand\thefootnote{}\footnote{#1}%
	\addtocounter{footnote}{-1}%
	\endgroup
}

\def\C          {\mathbb C}

\def\Cat         {\mathcal{C}}
\def\Dat         {\mathcal D}

\def\id          {\mathbf id}
\def\Z          {\mathbb Z}

\def\be                 {\begin{equation}}
\def\ee                 {\end{equation}}

\def\id               {{\mathrm{id}}}

\def\Hom         {\mathrm{Hom}}
\def\coev        {\mathrm{coev}}
\def\ev          {\mathrm{ev}}

\def\mod          {\mathrm{mod}}
\def \Nat        {\mathrm{Nat}}

\def \F           {\mathcal{F}}
\def \U           {\mathcal{U}}
\def \Hmod         {H\mathrm{-mod}}
\def \X            {\mathsf{X}}
\def \Y             {\mathsf{Y}}

\numberwithin{equation}{section}

\begin{document}
\numberwithin{equation}{section}
\title{\textbf{\Large{Davydov-Yetter cohomology, comonads\\ and Ocneanu rigidity}}}
\date{\emph{\small{$\dagger$ Institut Denis Poisson, CNRS, Universit\'e de Tours, Universit\'e d'Orl\'eans,\\ Parc de Grandmont, 37200 Tours, France} \\
 \small{$\S$ Dahlem Center for Complex Quantum Systems, Freie Universit{\"a}t Berlin, 14195 Berlin, Germany}\\
  \small{$\ddagger$ Fachbereich Mathematik, Universit\"at Hamburg, Bundesstra\ss e 55, 20146 Hamburg, Germany}}}
\author{Azat M. Gainutdinov$^{\dagger}$, Jonas Haferkamp$^{\S}$ and Christoph Schweigert$^{\ddagger}$}
\maketitle
\begin{abstract}\noindent
	Davydov-Yetter cohomology classifies infinitesimal deformations of tensor categories and of tensor functors. Our first result is that Davydov-Yetter cohomology for  finite tensor categories is equivalent to the  cohomology of a comonad arising from the central Hopf monad. This has several applications: First, we obtain a short and conceptual proof of Ocneanu rigidity. Second, it allows to use standard methods from comonad cohomology theory to compute Davydov-Yetter cohomology for a family of non-semisimple finite-dimensional Hopf algebras generalizing Sweedler's four dimensional Hopf algebra.
\end{abstract}

\blfootnote{
\mbox{}\quad\;\,
\textit{Keywords:} Finite tensor categories, Davydov-Yetter cohomologies, Comonad cohomologies,\\
\mbox{}\qquad\qquad\qquad\;\; Deformation theory, Hopf algebras.
}

\tableofcontents

\section{Introduction}

Tensor categories are ubiquitous in many problems in algebra, representation theory, quantum topology and  mathematical physics.
Considerable effort was spent to better understand their properties, 
especially for the subclass of \textit{fusion categories} over the field of complex numbers (see e.g.~\cite{ENO}), which are semisimple finite tensor categories. 
In particular, there is only a finite number of fusion categories (up to equivalence) corresponding to a fusion ring and only a finite number of braidings for a given fusion category.
 This is a consequence of  the so-called Ocneanu rigidity, the fact that fusion categories   admit only trivial  deformations of their monoidal structure.
 
 In contrast to fusion categories, \textsl{non-semisimple} finite tensor categories are much less  understood. The main motivation for this paper is to have a better understanding of the deformation theory of such categories and of tensor functors between them.
We recall that  infinitesimal deformations of tensor categories and tensor functors  are controlled by Davydov-Yetter (DY) cohomology, see~\cite{CY,Da,Y1,Y2} or  in this text  Definition~\ref{definition:DYcoh}, which is the cohomology  of a complex associated to a tensor functor $F\colon\Cat\to\Dat$, and will  be denoted by $H_{DY}^{\bullet}(F)$. In particular, the third Davydov-Yetter cohomology group of the identity functor on a tensor category $\Cat$ classifies infinitesimal deformations of the associator up to an equivalence. Infinitesimal deformations for the monoidal structure of tensor functors are classified by the second DY cohomology group of the respective functor. 
 Deformations of braidings in $\Cat$ can be also studied via deformations of appropriate tensor functors from $\Cat\times \Cat$ to $\Cat$, see details in~\cite[Thm.~2.18]{Y1}.
 
 For tensor functors $F$ between (multi)-fusion categories, we have the following vanishing theorem
 $$
 H^n_{DY}(F)=0\ , \qquad \text{for all} \quad  n>0.
 $$
  This immediately implies the abscence of infinitesimal deformations. This fact is known as Ocneanu rigidity and it is proven in \cite[Sec.~7]{ENO} using weak Hopf algebras.
 
We know that for non-semisimple categories the Ocneanu rigidity in the above form can not hold in general. 
Indeed, the theorem~\cite[Thm.\,6.3 \& Rem.\,6.4]{EG} provides examples of finite tensor categories with infinitely many non-equivalent rigid tensor structures.
The failure of  the Ocneanu rigidity in the non-semisimple case is also easy to see in the case of the forgetful functor for Hopf algebras. 
Let $H$ be a finite-dimensional Hopf algebra over a field~$k$, $\Hmod$ the finite tensor category of finite dimensional $H$-modules and $F$ the forgetful functor.
Then in this case, the groups $H^n_{DY}(F)$ are isomorphic to the $n$th Hochschild cohomologies $\mathrm{HH}^n(H^*,k)$ of the dual Hopf algebra $H^*$.
  The latter are the extension groups $\mathrm{Ext}^n_{H^*}(k,k)$,
and there are indeed  many examples where these groups are nonzero, e.g.\ for Sweedler's four dimensional Hopf algebra.
For other functors like the identity functor -- the case we are mostly interested in -- a direct calculation of $H^n_{DY}(\id)$ is quite involved.
 In the case of representations of groups over finite fields there are calculations for $n=2$, see~\cite{Da2,EG2}, and  in this paper
  we  provide another example (based on finite-dimensional Hopf algebras) that shows the DY cohomologies for the identity functor in all degrees can not be in general zero. 

A key result of this paper is a reformulation of the DY cohomology theory  via a more classical comonad cohomology theory~\cite{BB}. The advantage of such a reformulation is that we can use then  standard results from the comonad cohomology theory to prove useful properties of  DY cohomologies and even to provide explicit calculations in the Hopf algebra cases.
For a finite tensor category $\Cat$ and $F=\id_\Cat$,
the comonad $G$ in question is an endofunctor on the Drinfeld center $\mathcal{Z}(\Cat)$ of $\Cat$ constructed via the adjunction $\F\dashv \U$ where 
$\U\colon\mathcal{Z}(\Cat)\to\Cat$ is  the forgetful functor  and
  $\F\colon\Cat\to \mathcal{Z}(\Cat)$ is the free functor, i.e.\ $G=\F\circ \U$.
We prove that the DY cohomology of $\Cat$ is equivalent to the comonad cohomology of $G$.
 This is formulated in Theorem~\ref{theorem:Main}  for general (exact) tensor functors $F$.

The above adjunction also defines the corresponding monad $Z=\U\circ \F$
 on $\Cat$ that can be realized via the coend
 \begin{equation}
 Z(V):=\int^{X\in\Cat}X^{\vee}\otimes V\otimes X\ ,
 \end{equation}
 and the free functor $\F$ 
  is then just the induction functor corresponding to the monad $Z$.

We also note that $Z$ is  the well known central Hopf monad~\cite{DS,BV2,Sh2}, and when applied to the tensor unit $I$, $Z(I)$ is the canonical Hopf algebra object in $\Cat$ if $\Cat$ is  braided. 
This algebra  was also a central object of studies in understanding fundamental properties of factorizable tensor categories, e.g.\ in the mapping class group representations~\cite{Ly,KL,Sh1,FS,GR} associated to $\Cat$.

  Comonad cohomology theory for a comonad $G$  has properties similar to standard homological algebra. For example, a variant of the comparison theorem (or fundamental lemma) of homological algebra holds for any additive category and ``coefficient'' functors, for details see~\cite{BB,B} or in this text Theorem~\ref{theorem:Comparison}. This theorem is a major tool for computation of cohomology groups. The only difference from the standard homological algebra is that one replaces the notions of projectiveness and exactness by the notions of $G$-projectiveness and $G$-exactness, respectively, see Definition~\ref{definition:G-projective}. The comonad cohomology of $G$-projective objects -- similar to projective objects in homological algebra -- always vanishes (Proposition~\ref{theorem:G-acyclicity}). Combined with the reformulation of Davydov-Yetter cohomology in Theorem~\ref{theorem:Main}, this fact implies a short and conceptual proof of Ocneanu rigidity for fusion categories and their tensor functors, see Corollary~\ref{cor:ocneanu}. 
 More precisely, we first introduce a more general formulation of Davydov-Yetter cohomology where the coefficients (Definition~\ref{definition:DY}) are objects in the Drinfeld center, and then show that all these coefficients are $G$-projective, and thus the cohomology groups in   positive grades vanish.
 
In Section~\ref{sec:RepresentationCategories}, we consider the special case of finite tensor categories that are representation categories of finite dimensional Hopf algebras.
In Section~\ref{ssec:CentralH-mod}, we  describe  the  comonad $G$  for the case $F=\id$ and its  bar resolution, then 
 in Section~\ref{ssec:G-proj-Hopf}  we describe $G$-projective modules in Hopf algebraic terms and relate them to $H^*$ projectiveness, see Corollaries~\ref{cor:DH-G-proj} and~\ref{cor:G-proj-H}.
 In Section~\ref{ssec:forgetful}, we show how to reformulate the Davydov-Yetter cohomology of the forgetful functor as the Davydov-Yetter cohomology of the identity functor with non-trivial coefficients (Theorem~\ref{cor:HochschildasCoefficients}). 

As we briefly discussed above, the Ocneanu rigidity does not hold for non-semisimple finite tensor categories: there are examples of such 
categories with non-trivial DY cohomology. 
In general, these can hint towards finite deformations and, thus, be an indispensable tool to study continuous families of tensor categories.
 In particular, Section~\ref{ssec:example} is concerned with a family of non-semisimple  Hopf algebras over the field $\mathbb{C}$ of complex numbers that generalize Sweedler's four dimensional Hopf algebra: the so-called bosonization of the $k$-dimensional commutative super Lie algebra  $\Lambda \mathbb{C}^k$ which is  $B_k:=\Lambda \mathbb{C}^k\rtimes \mathbb{C}[\Z_2]$. We apply our reformulation of the DY cohomology as the comonad cohomology for the case of  $B_k\mathrm{-mod}$ -- the category of finite dimensional modules over $B_k$.
The only technical part is a construction of a $G$-projective resolution which is $G$-exact, with the final result  (see Theorem~\ref{proposition:DYSweedler})
 \begin{equation}
 \dim H_{DY}^n(B_k\mathrm{-mod})=
\begin{cases}
 0&\textrm{for $n$ odd}\\
 {k+n-1\choose n}& \text{for $n$ even},
 \end{cases}
 \end{equation}
which turned out to agree with $\dim H^n_{DY}(\mathcal{U}_{B_k})$, where $\U_{B_k}\colon B_k\mathrm{-mod}\to \mathrm{Vec}_{\C}$ is the forgetful functor. 
These results are to the best of our knowledge the first known examples of finite tensor categories where (non-trivial) Davydov-Yetter cohomology of the identity functor was computed in all degrees.

\subsubsection*{Acknowledgements}

We thank Alexey Davydov, Pavel Etingof and Jan-Ole Willprecht
 for helpful discussions.
AMG is supported by CNRS and  ANR project JCJC ANR-18-CE40-0001, and also thanks the Humboldt Foundation for a partial financial support.
 JH is funded by the Deutsche Forschungsgemeinschaft  (DFG, EI 519/14-1).
CS is partially funded by DFG, German Research
Foundation under Germany's Excellence Strategy -- EXC 2121 ``Quantum Universe''
 -- 390833306 and by the RTG 1670 ``Mathematics inspired by string theory and quantum field theory''. 

\section{(Co)monads and their cohomology theories}\label{ssec:comonads}
In this section, we recall some basic definitions about monads and then summarize results from \cite{BB} on the cohomology theory of comonads. Most of the material in this section is standard, and a reader familiar with the subject can skip it.

\subsection{Monads and comonads}
\begin{Definition}[Monads]
A \emph{monad} (sometimes called \emph{triple}) on a category $\Cat$ consists of the following data:
\begin{itemize}
	\item An endofunctor $T\colon\Cat\to\Cat$,
	\item a natural transformation \emph{unit} $\eta\colon\id\to T$ and
	\item a natural transformation \emph{multiplication} $\mu\colon T^2:=T\circ T\to T$.
\end{itemize}
These are subject to the following relations for all $X\in \Cat$:
\begin{equation*}
\begin{tikzpicture}
\matrix (m) [matrix of math nodes,row sep=3em,column sep=4em,minimum width=2em]
{
T^3(X)&  T^2(X)\\
T^2(X) & T(X) \\
};
\path[-stealth]

(m-1-1)  edge node [above] {\small $ T(\mu_X)$} (m-1-2)
(m-1-1) edge node [left] {\small $\mu_{T(X)}$} (m-2-1)
(m-2-1) edge node [below] {\small $\mu_X$} (m-2-2)
(m-1-2) edge node [right] {\small $\mu_X$} (m-2-2)
;
\end{tikzpicture}
\;\;\;\;\;\;
\begin{tikzpicture}
\matrix (m) [matrix of math nodes,row sep=3em,column sep=4em,minimum width=2em]
{
	T(X)&  T^2(X)&T(X)\\
	 & T(X) & \\
};
\path[-stealth]

(m-1-1)  edge node [above] {\small $\eta_{T(X)}$} (m-1-2)
(m-1-1) edge node [below left] {\small $\id$} (m-2-2)
(m-1-2) edge node [right] {\small $\mu_X$} (m-2-2)
(m-1-3) edge node [above] {\small $T(\eta_X)$} (m-1-2)
(m-1-3) edge node [below right] {\small $\id$} (m-2-2)
;
\end{tikzpicture}
\end{equation*}
A \emph{comonad} (sometimes called a \emph{cotriple}\footnote{See e.g.~\cite[Sec.\,8.6]{Weibel}}) 
$(G,\Delta,\varepsilon)$ is a functor $G\colon\Cat\to\Cat$ with natural transformations called \emph{counit} $\varepsilon\colon G\to \id$ and \emph{comultiplication} $\Delta\colon G\to G^2$. These have to satisfy the above diagrams with reversed arrows.
\end{Definition}
We need the notion of $T$-modules (which are sometimes also called $T$-algebras).
\begin{Definition}
	Given a monad on a category $\Cat$, the category $T\mathrm{-mod}$ of \emph{$T$-modules} consists of objects being pairs $(X,\beta_X)$ 
with $X\in \Cat$ and $\beta_X\colon T(X)\to X$, such that the following diagrams commute:
	\begin{equation}\label{eq:Tmod}
	\begin{tikzpicture}
	\matrix (m) [matrix of math nodes,row sep=3em,column sep=4em,minimum width=2em]
	{
		T^2(X)&  T(X)\\
		T(X) & X \\
	};
	\path[-stealth]
	
	(m-1-1)  edge node [above] {\small $\mu_X$} (m-1-2)
	(m-1-1) edge node [left] {\small $T(\beta_X)$} (m-2-1)
	(m-2-1) edge node [below] {\small $\beta_X$} (m-2-2)
	(m-1-2) edge node [right] {\small $\beta_X$} (m-2-2)
	;
	\end{tikzpicture}
	\;\;\;\;\;
	\begin{tikzpicture}
	\matrix (m) [matrix of math nodes,row sep=3em,column sep=4em,minimum width=2em]
	{
		X&  T(X)\\
		& X \\
	};
	\path[-stealth]
	
	(m-1-1)  edge node [above] {$\eta_X$} (m-1-2)
	(m-1-1) edge node [below left] {$\id$} (m-2-2)
	(m-1-2) edge node [right] {$\beta_X$} (m-2-2)
	;
	\end{tikzpicture}
	\end{equation}
	Furthermore, a morphism of $T$-modules $f\colon(X,\beta_X)\to(Y,\beta_Y)$ is a morphism $f\colon X\to Y$ in $\Cat$ such that the diagram
	\begin{equation*}
	\begin{tikzpicture}
	\matrix (m) [matrix of math nodes,row sep=3em,column sep=4em,minimum width=2em]
	{
		T(X)&  T(Y)\\
		X & Y \\
	};
	\path[-stealth]
	
	(m-1-1)  edge node [above] {\small $T(f)$} (m-1-2)
	(m-1-1) edge node [left] {\small $\beta_X$} (m-2-1)
	(m-1-2) edge node [right] {\small $\beta_Y$} (m-2-2)
	(m-2-1) edge node [below] {\small $f$} (m-2-2)
	;
	\end{tikzpicture}
	\end{equation*}
	commutes. $T\mathrm{-mod}$ is sometimes called the Eilenberg-Moore category of $T$.
\end{Definition}
In what follows,  for a $T$-module $(X,\beta_X)$ we will also use the  notation
\begin{equation}
 \mathsf{X}:=(X,\beta_X).
 \end{equation}
	\begin{Example}
	A simple example of a monad is provided by a monoid  $(A,m,u)$ in a monoidal category $\Cat$. The associated monad consists of the endofunctor
$T_A\colon\Cat\to\Cat$ such that $T_A(X)= A\otimes X$ and the natural transformations
	\begin{align}
	 \mu_X&=(m\otimes \id_X)\circ \alpha_{A,A,X}\colon A\otimes (A\otimes X)\to (A\otimes A)\otimes X\to A\otimes X,\\
 \eta_X&=u\otimes \id_X
 \colon X\to A\otimes X,
	 \end{align}
where $\alpha$ denotes the associator of $\Cat$.
	 Analogously, to every comonoid in a monoidal category one can associate a comonad.
\end{Example}
A source of monads and comonads are pairs of adjoint functors. More precisely, given a pair of adjoint functors $\F\dashv \U$, with $\F\colon\Cat\to \Dat$ (left adjoint) and $\U\colon\Dat \to\Cat$ (right adjoint) and unit $\eta\colon\id_{\Cat}\to \U\circ \F$ and counit $\varepsilon\colon\F\circ \U\to \id_{\Dat}$ of the adjunction, then $T:=\U\circ \F$ admits a canonical structure of a monad on $\Cat$ and $G:= \F\circ \U$ admits a canonical structure of a comonad on $\Dat$. Here, unit and counit of the monad $T$ and comonad $G$ are
$\eta$ and $\varepsilon$, respectively. 
The corresponding multiplication and comultiplication are defined as
\begin{align}
&\mu\colon\; T^2=\U\circ \F\circ \U\circ \F\xrightarrow{\U(\varepsilon_{\F(?)})}\U\circ \F=T,\nonumber\\
&\Delta\colon\; G= \F\circ \U\xrightarrow{\F(\eta_{\U(?)})} \F\circ \U\circ \F\circ \U=G^2.
\end{align} 
 However, given a monad $T$ on a category $\Cat$, there is usually more than one way to construct a pair of adjoint functors such that $T$ is induced by this adjunction. 
 The adjunction corresponding to $T$ is defined via the forgetful functor 
\begin{equation}\label{eq:U}
\mathcal{U}\colon T\mathrm{-mod}\to \Cat,\qquad \mathcal{U}(\mathsf{X}):=X,\qquad \mathcal{U}(f):=f
\end{equation}
 and the free functor 
\begin{equation}\label{eq:F}
\mathcal{F}\colon\Cat\to T\mathrm{-mod}, \qquad
\mathcal{F}(X):=(T(X),\mu_{X}),\qquad \mathcal{F}(f):= T(f).
\end{equation}
Then we have $T=\mathcal{U}\circ \mathcal{F}$. In the following, denote by
\begin{equation}\label{eq:G_T}
G_T:= \mathcal{F}\circ \mathcal{U}
\end{equation}
 the associated comonad on $T\mathrm{-mod}$. Notice that for $(X,\beta_X)\in T\mathrm{-mod}$
 $$
 G_T\colon (X,\beta_X) \mapsto (T(X),\mu_X)\qquad  \text{and} \qquad G_T^2\colon (X,\beta_X)\mapsto \left(T^2(X),\mu_{T(X)}\right).
 $$
  Then, the comultiplication and counit of $G_T$ are given on components by
 \begin{align}
\Delta_{\mathsf{X}}
&\colon\left(T(X),\mu_X\right)\xrightarrow{T(\eta_{X})}\left(T^2(X),\mu_{T(X)}\right),\nonumber\\
\varepsilon_{\mathsf{X}}
&\colon(T(X),\mu_X)\xrightarrow{\beta_X}(X,\beta_X).\label{G-eps-def}
\end{align}

\subsection{$G$-projective objects}
Here, we discuss the notion of $G$-projective that is needed later for the comonad cohomology theory.
\begin{Definition}\label{definition:G-projective}
	Let $(G,\Delta,\varepsilon)$ be a comonad on an additive category $\Cat$.
An object $X\in\Cat$ is called \emph{$G$-projective} if there exists a morphism $s\colon X\to G(X)$ in $\Cat$, called a section, such that $\varepsilon_X\circ s=\id_X$.
\end{Definition}

The following lemma yields a criterium to identify $G$-projective objects.

\begin{Lemma}\label{lemma:directsummands}
	Let $(G,\Delta,\varepsilon)$ be a comonad on an additive category $\Cat$. The following statements hold:
	\begin{enumerate}
		\item  Every object of the form $G(X)$ for some $X\in\Cat$ is $G$-projective.
		\item Direct summands of $G$-projective objects are $G$-projective.
	\end{enumerate}
\end{Lemma}
\begin{proof}
	By definition of a comonad, we have  $\varepsilon_{G(X)}\circ \Delta_X=\id_{G(X)}$. This already proves the first statement.
	To prove the second one, let $X\oplus Y$ be $G$-projective, with $X,Y\in \Cat$, i.e. there is a morphism $s\colon X\oplus Y\to G(X\oplus Y)$
	such that $\varepsilon_{X\oplus Y}\circ s=\id_{X\oplus Y}$. Recall that the counit $\varepsilon \colon G\to \id$ is a natural transformation. Denote the canonical embedding of $X$ into $X\oplus Y$ with $i_X\colon X\to X\oplus Y$ and the canonical projection onto $X$ with $p_X\colon X\oplus Y\to X$. Then it follows that
	\begin{equation}
	\varepsilon_X\circ G(p_X)\circ s\circ i_X=p_X\circ \varepsilon_{X\oplus Y}\circ s\circ i_X=p_X\circ i_X=\id_X,
	\end{equation}
	where the first equality holds because $\varepsilon$ is a natural transformation. Thus, $X$ is $G$-projective with the section $s_X=G(p_X)\circ s\circ i_X$.
\end{proof}

Using Lemma~\ref{lemma:directsummands} and Definition~\ref{definition:G-projective} of $G$-projective objects we get the corollary:
\begin{Cor}\label{cor:G-proj}
Let $(G,\Delta,\varepsilon)$ be a comonad on an additive category $\Cat$. An object $X$ is $G$-projective if and only if it is a retract of $G(Y)$ for some $Y$, i.e.\ if $X$ can be realised as a direct summand in $G(Y)$.
\end{Cor}

The next lemma provides further examples of $G$-projective objects.
\begin{Lemma}\label{lemma:projectiveisGprojective}
	Given an adjunction $\F\dashv\U$ defining a comonad $G$ on $\Dat$. If the right adjoint $\U$ is faithful, then every projective object in $\Dat$ is also $G$-projective.
\end{Lemma}
\begin{proof}
	 Recall that $G$ is equipped with a counit $\varepsilon\colon G\to \id$. We need the technical fact that $\varepsilon_X$ is an epimorphism for every $X\in \Dat$: It follows from \cite[Sec.~IV.3, Thm.~1]{M} that the counit $\varepsilon$ is component-wise an epimorphism if the right adjoint of the involved adjunction is faithful.
	This allows us to use the lifting property of a projective object $P\in \Dat$ to lift $\id_P\colon P\to P$ to $s_P\colon P\to G(P)$ such that $\varepsilon_P\circ s_P=\id_P$:
	\begin{equation*}
	\begin{tikzpicture}
	\matrix (m) [matrix of math nodes,row sep=3em,column sep=4em,minimum width=2em]
	{
		& P \\
		G(P) & P \\
	};
	\path[-stealth]
	
	(m-1-2) edge node [right] {\small $\id_P$} (m-2-2)
	(m-2-1) edge node [below] {\small $\varepsilon_{P}$} (m-2-2)
	(m-1-2) edge [dashed] node [above, sloped] {\small $s_P$} (m-2-1)
	;
	\end{tikzpicture}
	\end{equation*}
	This is just the definition of $G$-projectiveness (compare Definition \ref{definition:G-projective}).
\end{proof}

\subsection{Comonad cohomology}
A comonad on an additive category gives rise to a cohomology theory via the  construction of~\cite{BB}. It uses the notion of $G$-exactness:

\begin{Definition}
	Let $(G,\Delta,\varepsilon)$ be a comonad on an additive category $\Cat$.
	\begin{itemize}
		\item A sequence $X\xrightarrow{i}Y\xrightarrow{j}Z$ in $\Cat$ is called \emph{$G$-exact} if $j\circ i=0$ and 
		\begin{equation}
		\Hom_{\Cat}(G(A),X)\to \Hom_{\Cat}(G(A),Y)\to \Hom_{\Cat}(G(A),Z)
		\end{equation}
		is exact for all $A\in \Cat$.
		\item A sequence
		\begin{equation}
		\ldots \to P_n\to\ldots \to P_1\to P_0\to X\to 0
		\end{equation}
		is called a \emph{$G$-resolution of $X$} if $P_i$ is $G$-projective for $i\geq 0$ and the sequence is $G$-exact.
	\end{itemize}
\end{Definition}

\begin{Definition}\label{definition:BarResolution}
Given a comonad $(G,\Delta, \varepsilon)$ on an additive category $\Cat$ and an object $X\in \Cat$, the following sequence in $\Cat$ is called the \emph{bar resolution of $X$ associated to $G$}:
\begin{equation}\label{eq:barresolution}
\ldots \xrightarrow{d_n} G^n(X)\xrightarrow{d_{n-1}}\dots \xrightarrow{d_{2}}G^2(X)\xrightarrow{d_1}G(X)\xrightarrow{d_0:=\varepsilon_X} X\to 0,
\end{equation}
where 
\begin{equation}\label{eq:barresolutiondelta}
d_n:=\sum_{i=0}^n(-1)^i G^{n-i}\left(\varepsilon_{G^i(X)}\right).
\end{equation}
Given an abelian category $\Dat$ and an additive functor $E\colon\Cat\to\Dat$, the \emph{homology of $X$ associated to $G$ with coefficients in $E$} is defined as the homology of the complex
\begin{equation}\label{eq:comonad-cpx}
\ldots \xrightarrow{E(d_n)} E(G^n(X))\xrightarrow{E(d_{n-1})}\ldots \xrightarrow{E(d_{2})}E(G^2(X))\xrightarrow{E(d_1)}E(G(X))\xrightarrow{} 0.
\end{equation}
We denote the cochain groups by $C_{n}(X,E)_G=E(G^{n+1}(X))$ and the corresponding homology groups by $H_{n}(X,E)_G$ with $n\geq 0$. Similarly, for an additive functor $E\colon\Cat^{op}\to\Dat$ we define cochain complexes and cohomology: $C^{\bullet}(X,E)_G$ and $H^{\bullet}(X,E)_G$.
\end{Definition}

We note that from this definition it follows that $H^{\bullet}(X,E)_G$ is functorial in the variable~$X$ (by using naturality of $d_n$) and in the variable $E$, as stated in~\cite[p.3]{BB}.

The following statement was proven in~\cite{BB}, and we give a proof for completeness.

\begin{Lemma}\label{lemma:barresolution}
	The bar resolution is a $G$-resolution.
\end{Lemma}
\begin{proof}
	Every object in the sequence (except possibly $X$) is $G$-projective by Lemma~\ref{lemma:directsummands} (1). 
It is $G$-exact as well, as can be seen as follows: For $A\in \Cat$ and for the complex of abelian groups
	\begin{equation}\label{eq:complextrivial}
	\dots \to\Hom_{\Cat}\left(G(A),G^{n+1}(X)\right)\xrightarrow{d^*_{n}} \Hom_{\Cat}\left(G(A),G^n(X)\right)\xrightarrow{d^*_{n-1}} \Hom_{\Cat}\left(G(A),G^{n-1}(X)\right)\to \dots,
	\end{equation}
	with $d_n^*(f)=d_n\circ f$, we define a family of maps
	\begin{equation}
	h_n\colon\Hom_{\Cat}\left(G(A),G^n(X)\right)\xrightarrow{} \Hom_{\Cat}\left(G(A),G^{n+1}(X)\right)
	\end{equation}
	via $h_n(f):=(-1)^{n}G(f)\circ \Delta_A$.
	A simple calculation shows that this is a homotopy contraction:
	\begin{align}
	\left(d^*_{n}\circ h_n+h_{n-1}\circ d^*_{n-1}\right)(f)=&\sum_{i=0}^{n}(-1)^{n+i}G^{n-i}\left(\varepsilon_{G^i(X)}\right)\circ G(f)\circ\Delta_A\nonumber\\
	&+(-1)^{n-1}G\left(\sum_{i=0}^{n-1}(-1)^{i}G^{n-1-i}\left(\varepsilon_{G^i(X)}\right)\circ f\right)\circ \Delta_A\nonumber\\
	=&(-1)^{2n}\varepsilon_{G^{n}(X)}\circ G(f)\circ \Delta_A\nonumber\\
	=&f\circ \varepsilon_{G(A)}\circ \Delta_A=f,
	\end{align}
	where the first equality in the last  line  is due to naturality of $\varepsilon$, while the last equality is by  the counit axiom of a comonad.
	 The existence of a homotopy contraction implies that the complex is quasi-isomorphic to the zero complex.
\end{proof}
\begin{Example}[Hochschild cohomology]
	Hochschild cohomology provides an example of a comonad cohomology. For an associative algebra $A$ over a commutative ring $k$, consider the adjunction for the forgetful functor $\mathcal{U}\colon A\otimes A^{op}\mathrm{-mod}\to k\mathrm{-mod}$ and its left adjoint. This adjunction yields a comonad on $A\otimes A^{op}\mathrm{-mod}$ that is defined as follows:
	\begin{equation}
	G(V):=A\otimes A^{op}\otimes V,
	\end{equation}
	with the counit $\varepsilon_V\colon a\otimes v \mapsto a.v$, for $a\in A\otimes A^{op}$ and $v\in V$. We also note that in this case a module is $G$-projective if and only if it is projective in $A\otimes A^{op}\mathrm{-mod}$.
	
	It is easy to check that the  bar resolution~\eqref{eq:barresolution} is the (standard)  bar resolution of the $A\otimes A^{op}$-module $X$, see also~\cite[Sec.\,8.6.12]{Weibel}.
	Therefore, applying the coefficient functor $\Hom_{A\otimes A^{op}}(?,M)$ for an $A\otimes A^{op}$-module $M$ to the bar resolution~\eqref{eq:barresolution} with $X=A$ and taking cohomology yields $\mathrm{Ext}^{\bullet}_{A\otimes A^{op}}(A,M)$ which is the Hochschild cohomology of $A$ with coefficients in $M$.
\end{Example}

The following statements are proven in \cite[Sec.~4.2 \& Sec.~4.3]{BB}.
\begin{Proposition}\label{theorem:G-acyclicity}
		Let $(G,\Delta,\varepsilon)$ be a comonad on an additive category $\Cat$. Given a $G$-projective object $P\in \Cat$, then $H^n(P,E)_G=0$ for all $n>0$ and all coefficient functors~$E\colon \Cat \to \Dat$ where $\Dat$ is abelian.
\end{Proposition}

The fundamental lemma of homological algebra also generalizes to comonad cohomology:

\begin{Theorem}[Comparison theorem]\label{theorem:Comparison}
	Given a $G$-projective complex (i.e.\ all objects except possibly $X$ are $G$-projective)
	\begin{equation}
	\dots P_1\to P_0\to X
	\end{equation}
	and a $G$-exact complex
	\begin{equation}
	\dots \to Y_1\to Y_0\to Y.
	\end{equation}
	Then, every morphism $f\colon X\to Y$ can be extended to a morphism of complexes
	\begin{equation}
	\begin{tikzpicture}
	\matrix (m) [matrix of math nodes,row sep=3em,column sep=4em,minimum width=2em]
	{
	\textcolor{white}{a}\dots & P_1& P_0& X& 0\\
	\textcolor{white}{a}\dots	& Y_1 & Y_0& Y& 0 \\
	};
	\path[-stealth]
	(m-1-1) edge node [left] {$$} (m-1-2)
	(m-2-1) edge node [left] {$$} (m-2-2)
	(m-1-2) edge node [left] {\small $f_1$} (m-2-2)
	(m-1-2)  edge node [above] {$$} (m-1-3)
	(m-1-3) edge node [left] {\small $f_0$} (m-2-3)
	(m-1-3) edge node [right] {$$} (m-1-4)
	(m-1-4) edge node [below] {$$} (m-1-5)
	(m-2-4) edge node [below] {$$} (m-2-5)
	(m-2-3) edge node [below] {$$} (m-2-4)
	(m-2-2) edge node [below] {$$} (m-2-3)
	(m-1-4) edge node [left] {\small $f$} (m-2-4)
	;
	\end{tikzpicture}
\end{equation}
All extensions are pairwise chain homotopic. In particular, different G-resolutions of the same object lead to isomorphic (co)homologies.
\end{Theorem}
For a given monad $T$ on $\Cat$,
we now consider the comonad $G_T$ on $T\mathrm{-mod}$ defined in~\eqref{eq:G_T}. Furthermore, we consider the special case where the contravariant coefficient functor is
 $$
 E= \Hom_{T\mathrm{-mod}}(?,\mathsf{Y}) , 
 $$
 for $\mathsf{Y}\in T\mathrm{-mod}$.
   Then, the  complex~\eqref{eq:comonad-cpx} admits a canonical reformulation. The following proposition was proven in the section ``nonhomogeneous complex'' of \cite[p.~19-21]{B}. 
  
  \begin{Proposition}\label{proposition:MonadComonad}
	Given an additive category $\Cat$, a monad $(T,\mu,\eta)$ on $\Cat$ and two $T$-modules $\mathsf{X}=(X,\beta_X)$ and $\mathsf{Y}=(Y,\beta_Y)$,  then the complex $C^{\bullet}\left(\mathsf{X},\Hom_{T\mathrm{-mod}}(?,\mathsf{Y})\right)_{G}$ for $G=G_T$ is isomorphic to the complex with the cochain groups $\Hom_{\Cat}(T^{n}(X),Y)$, with $n\geq 0$,
	 and  with the differential
	\begin{align}\label{eq:differentialT}
	\partial(f)&:= f\circ T^n\left(\beta_X\right)+\sum_{i=1}^n(-1)^i f\circ T^{n-i}\left(\mu_{T^{i-1}(X)}\right)+(-1)^{n+1}\beta_Y\circ T(f)\ ,
	\end{align}
	where $f\in\Hom_{\Cat}(T^n(X),Y)$.
\end{Proposition}

\begin{proof}[Sketch of proof]
	Recall from Definition~\ref{definition:BarResolution} the cochain groups
	\begin{equation}
	 C^{n}\bigl(\mathsf{X},\Hom_{T\mathrm{-mod}}(?,\mathsf{Y})\bigr)_{G}
	 =\Hom_{T\mathrm{-mod}}\bigl(G^{n+1}(\mathsf{X}),\mathsf{Y}\bigr).
	 \end{equation}
	We have the following isomorphism
	\begin{align}
	\Hom_{\Cat}\bigl(T^n(X),Y\bigr)& = \Hom_{\Cat}\bigl(T^n(X),\U(Y,\beta_Y)\bigr)\nonumber\\
	&\cong\Hom_{T\mathrm{-mod}}\bigl(\F(T^n(X)),(Y,\beta_Y)\bigr)\nonumber\\
	&=\Hom_{T\mathrm{-mod}}\bigl(\F\circ (\U\circ \F)\circ\ldots \circ (\U\circ \F(X)), (Y,\beta_Y)\bigr)\nonumber\\
	&=\Hom_{T\mathrm{-mod}}\bigl(\F\circ \U\circ \F\circ\ldots \circ \U\circ \F\circ \U (X,\beta_X), (Y,\beta_Y)\bigr)\nonumber\\
	&=\Hom_{T\mathrm{-mod}}\bigl((\F\circ \U)\circ\ldots \circ (\F\circ \U) (X,\beta_X), (Y,\beta_Y)\bigr)\nonumber\\
	&=\Hom_{T\mathrm{-mod}}\bigl(G^{n+1}(X,\beta_X), (Y,\beta_Y)\bigr)\nonumber\\
	&=C^n(\mathsf{X},\Hom_{T\mathrm{-mod}}(?,\mathsf{Y}))_G,
	\end{align}
	where the only non-trivial map is  the adjunction isomorphism, and  the last equality  is by definition of the cochain groups.
	 One can also check that the above isomorphism is a cochain map.
\end{proof}

\section{Davydov-Yetter cohomology as a comonad cohomology}\label{sec:DY}
In this section, we introduce Davydov-Yetter cohomology with coefficients, thereby generalizing the original notion~\cite{CY,Da,Y1,Y2}. We show that Davydov-Yetter cohomology can be reformulated as comonad cohomology of a generalization of the central Hopf monad (Theorem \ref{theorem:Main}). After providing a detailed proof, we showcase the power of this point of view with a short and conceptual proof of Ocneanu rigidity.

\subsection{Conventions}
Let $k$ denote a field and $\mathrm{Vec}_k$ is the category of finite dimensional $k$-linear vector spaces. A \emph{tensor category} will always mean a rigid, $k$-linear, abelian monoidal category such that the monoidal product is bilinear. We call a category \emph{finite} if it is $k$-linear and equivalent to the category of finite dimensional representations of a finite dimensional $k$-algebra. By a \emph{finite tensor category} we mean a tensor category which is finite as an abelian category.
Notice that we do not assume the tensor unit to be simple in contrast to e.g. \cite{EGNO} or \cite{ENO}. In fact, our definition of a finite tensor category is what is called a finite multi-tensor category in \cite{EGNO}.

Recall that a monoidal category $\Cat$ is called rigid if every object $V\in\Cat$ has a left dual $^{\vee}V$ and a right dual $V^{\vee}$ together with left and right (co)evaluation maps
\begin{align}
\ev_V\colon&V^{\vee}\otimes V\to I\qc \qquad\coev_V\colon I\to V\otimes V^{\vee},\\
\widetilde \ev_V\colon&V\otimes {^{\vee}V}\to I\qc\qquad \widetilde \coev_V\colon {^{\vee}V}\otimes V\to I,
\end{align}
satisfying the standard axioms.
We will use the following graphical notations:
\begin{align}\label{eq:ev-coev-gr}
\ev_V &~=\hspace*{.7em}   \ipic{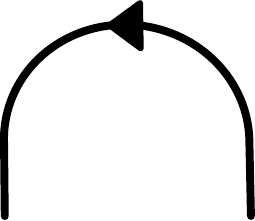}{.25}  
\put(-34,-22){\scriptsize $V^{\vee}$} \put(-4,-22){\scriptsize $V$} \qc&
\coev_{V} &~=\hspace*{.7em} \ipic{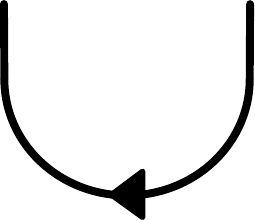}{.25}  
\put(-34,16){\scriptsize $V$} \put(-5,16){\scriptsize $V^{\vee}$}  \qc&\\ 
\widetilde\ev_V &~=\hspace*{.7em} \ipic{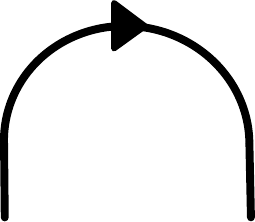}{.25} 
\put(-34,-22){\scriptsize $V$} \put(-6,-22){\scriptsize $^{\vee}V$} 
\qc
& 
\widetilde\coev_V &~=\hspace*{.7em} \ipic{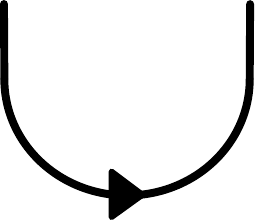}{.25} 
\put(-38,16){\scriptsize $^{\vee}V$}
\put(-5,16){\scriptsize $V$} 
\qp\notag
\end{align}
Here, string diagrams must be read upwards. 
General morphisms will be presented by coupons, see e.g. Remark~\ref{remark:graphicalDY}.

A tensor functor $F\colon\Cat\to\Dat$ between tensor categories is a $k$-linear monoidal functor, i.e.\ equipped with a natural isomorphism $\psi_{V,W}\colon F(V)\otimes F(W)\to F(V\otimes W)$ and an  isomorphism $\eta\colon I_{\Dat}\to F(I_{\Cat})$ satisfying the usual commuting diagrams.
 Often, if it follows from the context, we supress the subscript and use the notation $I$ for both monoidal units $I_{\Cat}$ and~$I_{\Dat}$.
Given a functor $F\colon\Cat\to\Dat$, we denote via 
\begin{equation}
F^{\times n}\colon\Cat\times\dots \times\Cat\to\Dat\times\dots \times\Dat, \qquad n\geq 0,
\end{equation}
 the functor that is defined by applying $F$ component-wise, and where $F^{\times 0}$ is the identity endofunctor on $\mathrm{Vec}_k$.
 We reserve $F^n$ for the composition $F\circ\dots \circ F$, assuming $\Cat=\Dat$. By slight abuse of this notation, we denote with 
 \begin{equation}
 \otimes^n\colon\Cat\times\dots \times \Cat\to \Cat
 \end{equation}
  the functor that acts on objects $X_1,\dots,X_n\in \Cat$ as 
  $$
  \otimes^n(X_1,\ldots,X_n)=X_1\otimes(X_2\otimes(\ldots\otimes X_n)\ldots),
  $$
for $n\geq 2$. Furthermore, we use the convention $\otimes^1=\id_{\Cat}$ and $\otimes^0:\mathrm{Vec}_k \to \Cat$ is the additive functor that sends the ground field $k$ to the tensor unit in $\Cat$.

As usual, we denote ends and coends via the integral notation, i.e. an end and a coend of a functor  
$J\colon\Cat^{op}\times \Cat\to\Dat$ are denoted respectively by 
\begin{equation}
 \int_{X\in\Cat}J(X,X)\;\;\;\;\;\;\text{and}\;\;\;\;\;\;\;\int^{X\in\Cat}J(X,X).
\end{equation}

\subsection{Davydov-Yetter cohomology with coefficients}\label{ssec:DY_coefficients}
Davydov-Yetter cohomology for a monoidal functor targeting a tensor category was developed in \cite{Y1} and \cite{Y2} based on work in \cite{CY} and independently in \cite{Da}. We will introduce the case of a more general complex with `coefficients'. These will be objects in the centralizer of a monoidal functor (compare also \cite[Sec.~3]{Sh2}).

\begin{Definition}
\label{def:half-br}
	Let $F\colon\Cat\to\Dat$ be a monoidal functor between monoidal categories and $X\in \Dat$. We say that a natural isomorphism $\rho^X\colon X\otimes F(?)\to F(?)\otimes X$ is a \emph{half-braiding relative to $F$} if the diagram
	\begin{equation}\label{eq:halfbraiding}
	\begin{tikzpicture}
	\matrix (m) [matrix of math nodes,row sep=3em,column sep=5em,minimum width=2em]
	{
		X\otimes F(V)\otimes F(W)& X\otimes F(V\otimes W)\\
		F(V)\otimes X\otimes F(W) &   \\
		F(V)\otimes F(W)\otimes X& F(V\otimes W)\otimes X \\
	};
	\path[-stealth]
	
	(m-1-1)  edge node [above] {\small $\id_X\otimes\psi_{V,W}$} (m-1-2)
	(m-1-1) edge node [left] {\small $\rho^X_V\otimes \id_{F(W)}$} (m-2-1)
	(m-2-1) edge node [left] {\small $\id_{F(V)}\otimes \rho^X_{W}$} (m-3-1)
	(m-3-1) edge node [below] {\small $\psi_{V,W}\otimes \id_X$} (m-3-2)
	(m-1-2) edge node [right] {\small $\rho^X_{V\otimes W}$} (m-3-2)
	;
	\end{tikzpicture}
	\end{equation}
	commutes for all $V,W\in \Cat$, where for simplicity we assumed that $\Dat$ is strict.
	\end{Definition}
\newcommand{\trivbr}{\sigma}

	\begin{Definition}
\label{def:centralizer}
 The \emph{centralizer} $\mathcal{Z}(F)$ of $F$ is the category where objects are pairs $(X,\rho^X)$ with $X$ an object in $\mathcal{D}$ and $\rho^X$ a half-braiding relative to $F$.
  The morphisms $f\colon(X,\rho^X)\to (Y,\rho^Y)$ are morphisms $f\colon X\to Y$ in $\Dat$ such that the diagram
	\begin{equation}\label{eq:diag-centrF-morphism}
	\begin{tikzpicture}
	\matrix (m) [matrix of math nodes,row sep=3em,column sep=4em,minimum width=2em]
	{
		X\otimes F(V)&  F(V)\otimes X\\
		Y\otimes F(V) & F(V)\otimes Y \\
	};
	\path[-stealth]
	
	(m-1-1)  edge node [above] {\small $\rho^X_V$} (m-1-2)
	(m-1-1) edge node [left] {\small $f\otimes \id_{F(V)}$} (m-2-1)
	(m-1-2) edge node [right] {\small $\id_{F(V)}\otimes f$} (m-2-2)
	(m-2-1) edge node [below] {\small $\rho^Y_V$} (m-2-2)
	;
	\end{tikzpicture}
	\end{equation}
	commutes for all $V\in\Cat$.
The special case of $\Cat=\Dat$ and $F=\id$ is called \emph{Drinfeld center of $\Cat$} and denoted by $\mathcal{Z}(\Cat)$.
\end{Definition}
It is well known that the category $\mathcal{Z}(F)$ admits the canonical structure of a monoidal category~\cite{Maj2,Sh2}. 
In particular, the tensor unit in $\mathcal{Z}(F)$ is $I=I_{\Dat}$ together with 
the half-braiding  
\begin{equation}
\trivbr_X\colon I\otimes F(X)\xrightarrow{\;\cong\;} F(X)\xrightarrow{\;\cong\;} F(X)\otimes I.
\end{equation}
We will denote  the tensor unit in 	$\mathcal{Z}(F)$ 
by $\mathsf{I}=(I,\trivbr)$.

From now on for brevity, we will supress coherence isomorphisms of monoidal categories and functors, that is, we work with strict monoidal categories and monoidal functors. 
\begin{Definition}
	[Davydov-Yetter complex]\label{definition:DY}
	Let $F\colon\Cat\to \Dat$  be a monoidal functor, where $\Cat$~is a monoidal category and $\Dat$ is a tensor category and let 
	$$
	\mathsf{X}=(X,\rho^X),\; \mathsf{Y}=(Y,\rho^Y)\in \mathcal{Z}(F).
	$$
	  The \emph{Davydov-Yetter complex of $F$ with coefficients $\mathsf{X}$ and  $\mathsf{Y}$} and denoted by $C_{DY}^{\bullet}(F,\mathsf{X},\mathsf{Y})$ consists of the following data:
	\begin{itemize}
		\item Cochain vector spaces for 
$n\geq 0$:
		\be
		C^{n}_{DY}(F,\mathsf{X},\mathsf{Y}):=\Nat\left(X\otimes(\otimes^n \circ F^{\times n}),
		(\otimes^n\circ F^{\times n})\otimes Y\right).
		\ee
		\item Differential
		\begin{align}\label{equation:DYdifferential}
		\delta^n(f)_{X_0,\dots,X_n}:= &\left(\id_{F(X_0)}\otimes f_{X_1,\dots,X_n}\right)\circ \left(\rho^X_{X_0}\otimes \id_{F(X_1)\otimes\dots \otimes F(X_n)}\right)+\nonumber\\
		&+\sum_{i=1}^n (-1)^if_{X_0,\dots,X_{i-1}\otimes X_i,\dots ,X_n}+\\
		&+(-1)^{n+1}\left(\id_{F(X_0)\otimes\dots \otimes F(X_{n-1})}\otimes \rho_{X_n}^Y\right)\circ \left(f_{X_0,\dots ,X_{n-1}}\otimes\id_{F(X_n)}\right)\nonumber.
		\end{align}
		\end{itemize}
Here, for $n=0$ the cochain spaces are $C^0_{DY}(F,\mathsf{X},\mathsf{Y})= \Hom_{\mathcal{D}}(X,Y)$,  recall   our conventions on $\otimes^0$ and $F^{\times 0}$, and the differential takes the form
\begin{equation}
\delta^0(f)_{X_0}:=\left(\id_{F(X_0)}\otimes f\right)\circ \rho^X_{X_0}- \rho^Y_{X_0}\circ \left(f\otimes\id_{F(X_0)}\right).
\end{equation}
		For the following complexes, we also use the  notations 
		 $$
		 C^{\bullet}_{DY}(F):=C^{\bullet}_{DY}(F,\mathsf{I},\mathsf{I})\ , \qquad
	C^{\bullet}_{DY}(\Cat,\mathsf{X},\mathsf{Y}):=C^{\bullet}_{DY}(\id_{\Cat},\mathsf{X},\mathsf{Y})\ , \qquad
		C^{\bullet}_{DY}(\Cat):=C^{\bullet}_{DY}(\id_{\Cat}),
		$$
		 and call them \emph{Davydov-Yetter complex of $F$},
		  and \emph{Davydov-Yetter complex of $\Cat$ with coefficients in $\mathsf{X}$ and $\mathsf{Y}$},
		  and \emph{Davydov-Yetter complex of $\Cat$}, respectively.
\end{Definition}
The fact that the right hand side of~\eqref{equation:DYdifferential} is a natural transformation follows from naturality of $f$ and naturality of the half-braidings $\rho^X$ and $\rho^Y$.
It is also straightforward to check that $\delta^{n+1}\circ\delta^n=0$. The statement for trivial coefficients is well-known 
 \cite{Da,Y1}, while
  the general case follows by a very similar calculation and  using the half-braiding property~\eqref{eq:halfbraiding}.
\begin{Definition}[Davydov-Yetter cohomology]\label{definition:DYcoh}
	The cohomology of the cochain complex $C_{DY}^{\bullet}(F,\mathsf{X},\mathsf{Y})$ is called \emph{Davydov-Yetter cohomology}\footnote{We also use shorter \textit{DY cohomology}.} and denoted by 
	$$
	H^{\bullet}_{DY}(F,\X,\Y):=H^{\bullet}\bigl(C^{\bullet}_{DY}(F,\X,\Y)\bigr).
	$$
	 We denote the special cases by 
	 $$
	 H^{\bullet}_{DY}(F):=H^{\bullet}_{DY}(F,\mathsf{I},\mathsf{I})\ , \qquad 
	H^{\bullet}_{DY}(\Cat,\mathsf{X},\mathsf{Y}):=H^{\bullet}_{DY}(\id_{\Cat},\mathsf{X},\mathsf{Y})\ ,
	\qquad H^{\bullet}_{DY}(\Cat):=H^{\bullet}_{DY}(\id_{\Cat}).
	 $$
\end{Definition}
\begin{Remark}
In the non-strict version of \eqref{equation:DYdifferential}, the coherence isomorphisms of $\Cat, \Dat$ and $F$ can be inserted without much additional effort. For a formulation with coherence isomorphisms and trivial coefficients, we refer to \cite{Y1} and \cite{Y2}.
\end{Remark}
\begin{Remark}\label{remark:graphicalDY}
	The differential defining the Davydov-Yetter complex in \eqref{equation:DYdifferential} can be written using graphical notation:
	\begin{align}
	\delta^n(f)_{X_0,...,X_n}=&\;\; \ipic{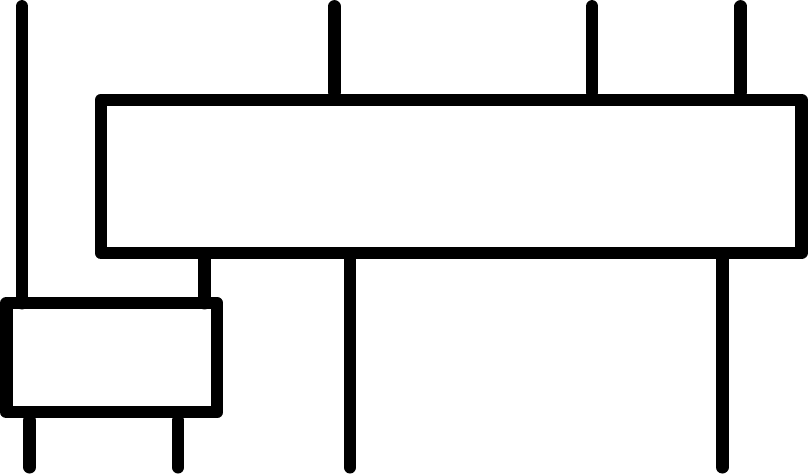}{0.25}\put(-98,-35){\tiny $X$}\put(-83,-35){\tiny $F(X_0)$}\put(-60,-35){\tiny $F(X_1)$}\put(-20,-35){\tiny $F(X_n)$}\put(-45,6){\scriptsize $f$}\put(-90,-16){\tiny $\rho^X_{X_0}$}\put(-103,30){\tiny $F(X_0)$}\put(-68,30){\tiny $F(X_1)$}\put(-38,30){\tiny $F(X_n)$}\put(-10,30){\tiny $Y$}\put(-45,23){\tiny $\ldots$}\put(-39,-15){\tiny $\ldots$}\;\;+\;\;
	\sum_{i=1}^n(-1)^i\;\ipic{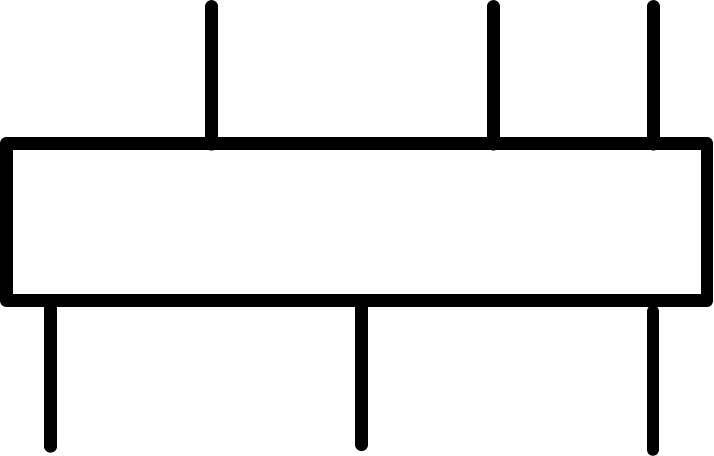}{0.25}\put(-82,-35){\tiny $X$}\put(-67,-35){\tiny $F(X_{i-1}\otimes X_i)$}\put(-15,-35){\tiny $F(X_n)$}\put(-50,0){\scriptsize $f$}\put(-70,30){\tiny $F(X_0)$}\put(-40,30){\tiny $F(X_n)$}\put(-10,30){\tiny $Y$}\put(-65,-20){\tiny $\ldots$}\put(-30,-20){\tiny $\ldots$}\put(-45,20){\tiny $\ldots$}\;\;\nonumber\\
	&+(-1)^{n+1}\;\ipic{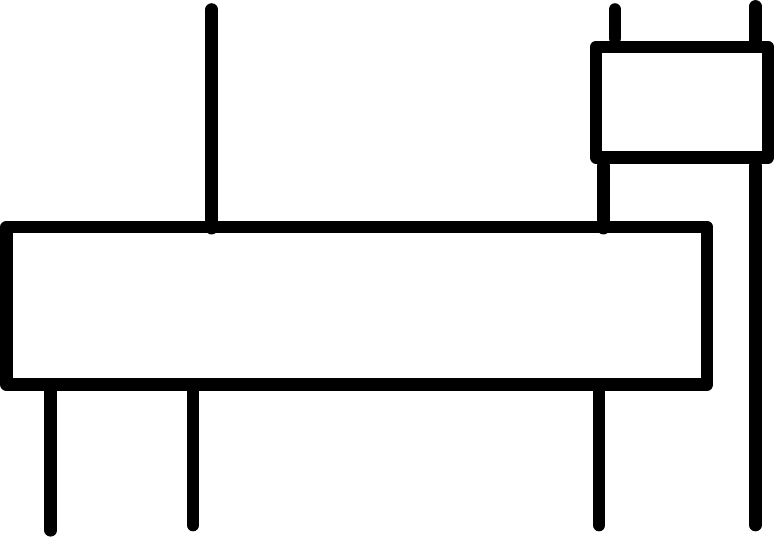}{0.25}\put(-90,-40){\tiny $X$}\put(-78,-40){\tiny $F(X_0)$}\put(-43,-40){\tiny $F(X_{n-1})$}\put(-10,-40){\tiny $F(X_{n})$}\put(-78,35){\tiny $F(X_0)$}\put(-5,35){\tiny $Y$}\put(-32,35){\tiny $F(X_n)$}\put(-50,-7){\scriptsize $f$}\put(-17,18){\tiny $\rho^Y_{X_n}$}\put(-55,-25){\tiny $\ldots$}\put(-50,15){\tiny $\ldots$} \ .
	\end{align}
\end{Remark} 
\begin{Remark}
	As it is often the case in cohomology theories, low degrees of Davydov-Yetter cohomology have concrete interpretations~\cite{CY,Da,Y1}. 
	In particular, 
\begin{itemize}
\item
 $H^0_{DY}(F, \mathsf{X},\mathsf{Y})$  consists of  those elements in $\Hom_{\mathcal{D}}(X,Y)$ which are also morphisms in the \emph{centralizer} $\mathcal{Z}(F)$, recall~\eqref{eq:diag-centrF-morphism};
\item
$H^1_{DY}(F)$  consists of   derivations of $F$: $\eta\in \Nat(F,F)$ such that 
$$
\eta_{X\otimes Y} = \eta_X\otimes \id + \id\otimes \eta_Y
$$
 modulo the inner derivations of $F$. By inner derivations here we mean those derivations $\eta$ that can be written as $\eta_X = f\otimes \id_{F(X)} - \id_{F(X)}\otimes f$ for some $f\in \mathrm{End}_{\Dat}(I_\Dat)$;
\item 
	 $H^2_{DY}(F)$ classifies first order infinitesimal deformations of the monoidal structure of~$F$ up to equivalence.
	 Obstructions to  extensions of them to finite deformations live in $H^3_{DY}(F)$;
	 \item
$H^3_{DY}(\Cat)$ classifies up to equivalence first order infinitesimal	
deformations of the associator of a tensor category~$\Cat$,  and obstructions are controlled by $H^4_{DY}(\Cat)$.
\end{itemize}
\end{Remark}

\subsection{The central monad and its variants}\label{ssec:CentralMonad}
Let $F\colon\Cat\to \Dat$ be a strict monoidal functor between strict rigid monoidal categories $\Cat$ and~$\Dat$. If for every $V\in \Dat$ the object
\begin{equation}\label{eq:ZF-def}
Z_{F}(V):=\int^{X\in\Cat}F(X)^{\vee}\otimes V\otimes F(X)
\end{equation}
exists, 
then the functor $Z_F(?)\colon \Dat\to \Dat$ has the natural structure of a monad~\cite{DS,Sh2}.
 Indeed, let 
\begin{equation}\label{iF-def}
 i^F_X(V)\colon \;F(X)^{\vee}\otimes V\otimes F(X)\to Z_F(V)
\end{equation}
 denote the universal dinatural transformation associated to $V\in\Dat$. We know from the Fubini theorem for coends~\cite[Prop.~IX.8]{M} that 
 the object
 \be
Z_F^2(V):=(Z_F\circ Z_F)(V) = \int^{(X,Y)\in\Cat\times\Cat}F(Y)^{\vee}\otimes F(X)^{\vee}\otimes V\otimes F(X) \otimes F(Y)
\ee 
exists and is a coend  with the universal dinatural transformation 
$$
 i^{(2)}_{(X,Y)}(V)\colon \; (FY)^{\vee}\otimes (FX)^{\vee}\otimes V\otimes FX\otimes FY\to Z_F^2(V)
$$
defined as
\be\label{eq:i2}    
  i^{(2)}_{(X,Y)}(V)= i^F_{Y}\bigl(Z_F(V)\bigr)\circ\bigl(\id_{(FY)^{\vee}}\otimes i^F_{X}(V)\otimes \id_{FY}\bigr),
\ee
where for brevity we replace $F(X)$ by $FX$, etc.
Recall that $F$ is a (strict) monoidal functor, therefore we have  the dinatural transformation 
\begin{equation}
i^F_{X\otimes Y}(V)\colon \; (FY)^{\vee}\otimes (FX)^{\vee}\otimes V\otimes FX\otimes FY\to Z_F(V).
\end{equation}
 Then, the multiplication for $Z_F$ is defined as the unique family of morphisms 
 $$
 \mu^F_V\colon Z_F^2(V)\to Z_F(V)
 $$
  such that 
 \begin{equation}\label{eq:muF-def}
 \mu^F_V\circ i^{(2)}_{(X,Y)}(V)=i^F_{X\otimes Y}(V).
 \end{equation}
  Furthermore, the unit is defined as
 \be\label{eq:def-eta}
 \eta^F_V\colon V\to Z_F(V), \quad \eta^F_V:=i^F_{I_{\Dat}}(V).
 \ee
 
\begin{Definition}
	The above defined monad $(Z_F,\mu^F,\eta^F)$ is called the \textit{central monad of the monoidal functor $F$}. 
\end{Definition}
\begin{Remark}
 For $F=\id$, we denote $(Z,i):=(Z_{\id}, i^{\id})$. This special case is called \textit{the central monad of the category $\Cat$}.
\end{Remark}
The central monad always exists for exact tensor functors $F\colon\Cat\to\Dat$ between finite tensor categories.
 This follows from the following fact proven in \cite[Cor.~5.1.8.]{KL}:
Let $\Cat$ and $\Dat$ be finite $k$-linear, abelian categories and $J\colon\Cat^{op}\times\Cat\to\Dat$ a functor that is $k$-linear and exact in each variable, then the coend $\int^{X\in\Cat}J(X,X)$ exists. 
Thus, for $J(X,Y)= F(X)^{\vee}\otimes V\otimes F(Y)$ we obtain that $Z_F$ exists.

The monad $Z_F$ can be further equipped with the structure of a bimonad. We recall that a monad $T$ is called \textit{bimonad}
if it admits  a natural transformation $\Psi_{V,W}\colon T(V\otimes W)\to T(V)\otimes T(W)$ and a morphism $\alpha\colon T(I)\to I$
satisfying axioms of a comonoidal functor  (for details, see e.g.~\cite[Sec.~2]{BV1}).
A bimonad structure on a $k$-linear endofunctor $T$ is equivalent to   the structure of a $k$-linear monoidal category on $T\mathrm{-mod}$ such that the forgetful functor $T\mathrm{-mod}\to \Cat$ is strict monoidal. 
Here, the tensor unit is $(I,\alpha)$ and it will  be denoted by  $\mathsf{I}$.
For $T=Z_F$, the structural morphism $\alpha\colon Z_F(I)\to I$ that we will denote by $\alpha^F$ is the unique morphism satisfying
\begin{equation}\label{eq:alpha}
\alpha^F\circ i^F_X(I):= \mathrm{ev}_{F(X)}.
 \end{equation}
 The comultiplication $\Psi^F$ for $Z_F$ is the unique natural transformation fixed by
 \begin{equation}\label{eq:def_comultiplication}
 \Psi^F_{V,W}\circ i^F_X(V\otimes W)=\left(i^F_X(V)\otimes i^F_X(W)\right)\circ \left(\id_{(FX)^{\vee}}\otimes \id_V\otimes \coev_{FX}\otimes \id_W\otimes \id_{FX}\right).
 \end{equation}
Furthermore, $Z_F\mathrm{-mod}$ is rigid~\cite{Sh2} and thus $Z_F$ is a \emph{Hopf monad}~\cite{BV1}.

From here on, we will supress the superscript in the structural maps if the functor $F$ is clear from the context.

The central Hopf monad $Z_F$ of $F$ is closely related to the centralizer $\mathcal{Z}(F)$ from Definition~\ref{def:centralizer}. 
The following can be found in~\cite[Thm.~5.12]{BV2} for $F=\id$ and for general case in~\cite[Lem.~3.3]{Sh2} .

\begin{Proposition}\label{theorem:equivalencecentralizer}
	Let $F\colon\Cat\to \Dat$  be a tensor functor between finite tensor categories such that $Z_F$ exists. Its centralizer $\mathcal{Z}(F)$ is  isomorphic as a tensor category to $Z_F\mathrm{-mod}$.
\end{Proposition}

We summarize the construction of the isomorphism from Proposition~\ref{theorem:equivalencecentralizer}, given in~\cite{BV2} in the case $F=\id$. 
Given a pair
 $(M,\rho)\in \mathcal{Z}(F)$ 
 with $M\in\Dat$ and a half-braiding $\rho_X\colon M\otimes F(X)\to F(X)\otimes M$. Then, the following diagram
\begin{equation}\label{equation:Module}
\begin{tikzpicture}
\matrix (m) [matrix of math nodes,row sep=3em,column sep=15em,minimum width=2em]
{
	FX^{\vee}\otimes M\otimes FX& FX^{\vee}\otimes FX\otimes M \\
	Z_F(M) & M \\
};
\path[-stealth]

(m-1-1)  edge node [above] {\small $\id_{FX^{\vee}}\otimes \rho_X$} (m-1-2)
(m-1-2) edge node [right] {\small $\ev_{FX}\otimes \id_M$} (m-2-2)
(m-1-1) edge node [left] {\small $i_X(M)$} (m-2-1)
(m-2-1) edge [dashed] node [below] {\small $!\exists \beta$} (m-2-2)
;
\end{tikzpicture}
\end{equation}
defines a unique morphism $\beta\colon Z_F(M)\to M$ due to universality of the coend $Z_F(M)$.
It is straightforward to prove that $(M,\beta)$ is in $Z_F\mathrm{-mod}$. In particular, to check~\eqref{eq:Tmod}, which is $\beta\circ Z_F(\beta)= \beta\circ \mu^F_M$, it is enough to precompose both sides by $i^{(2)}_{(X,Y)}(M)$ and apply definitions of structural maps of $Z_F$.

 On the other hand, given a $Z_F$-module structure $\beta\colon Z_F(M)\to M$, it can be shown
 that the following defines a half-braiding on~$M$:
\begin{equation}\label{eq:iso-braiding}
\rho_X\colon
M\otimes FX\xrightarrow{\mathrm{coev}_{FX}\otimes\id}FX\otimes (FX)^{\vee} \otimes M\otimes FX\xrightarrow{\id\otimes i_X(M)}
FX\otimes Z_F(M)\xrightarrow{\id\otimes\beta}FX\otimes M.
\end{equation}
We note that the inverse to this half-braiding is
\begin{align*}
\rho_X^{-1} = 
 FX\otimes M\xrightarrow{\id \otimes{\widetilde\coev}_{FX}} FX \otimes M \otimes {^{\vee}\!(FX)}\otimes FX
 & \xrightarrow{\id\otimes\rho_{{^{\vee}\!(FX)}}\otimes\id}  FX \otimes {^{\vee}\!(FX)}\otimes M \otimes FX\\
&\xrightarrow{{\widetilde\ev}_{FX}\otimes\id} M \otimes FX.
\end{align*}

\medskip

As described in Section~\ref{ssec:comonads}, $Z_F$ can be obtained from an adjunction consisting of the forgetful functor $\U_F\colon Z_F\mathrm{-mod}\to\Dat$ and the free functor $\F_F\colon \Dat\to Z_F\mathrm{-mod}$ such that 
\begin{equation}
Z_F=\U_F\circ \F_F.
\end{equation}
 The associated comonad $G_{Z_F}$ on $Z_F\mathrm{-mod}$ as defined in \eqref{eq:G_T} will be denoted for brevity by
\begin{equation}\label{eq:centralcomonad}
G_{F}:=\F_F\circ \U_F.
\end{equation}

This allows us to formulate the following theorem: Davydov-Yetter cohomology of a tensor functor $F$ can be reformulated as the cohomology of the comonad $G_F$, provided that the comonad $G_F$ exists. In particular, this is the case for finite tensor categories and exact functors between them.

\begin{Theorem}\label{theorem:Main} Let $\Cat$ and $\Dat$ be tensor categories and $F\colon\Cat\to\Dat$
  a tensor functor  such that the functor $Z_F$ exists. Furthermore, let $\X=(X,\rho^X),\Y=(Y,\rho^Y) \in \mathcal{Z}(F)$. 
  Then, the Davydov-Yetter complex $C^{\bullet}_{DY}(F,\X,\Y)$ from Definition~\ref{definition:DY} is isomorphic to the comonad complex $C^{\bullet}(\X,\Hom_{Z_F\mathrm{-mod}}(?,\Y))_{G_F}$ from Definition~\ref{definition:BarResolution}, for the comonad $G=G_F$ as defined in \eqref{eq:centralcomonad} and where $\X$ and $\Y$ are identified with the corresponding objects in $Z_F\mathrm{-mod}$ as in~\eqref{equation:Module}.
\end{Theorem}

We provide a proof below but first we note that the isomorphism of complexes in Theorem~\ref{theorem:Main} is a powerful tool for the computation of Davydov-Yetter cohomology as will be demonstrated in Section~\ref{ssec:Ocneanu} (Ocneanu rigidity) and in Section~\ref{ssec:example} in a class of examples of non-semisimple Hopf algebras. 
 A further advantage is that we obtain the following immediate corollary from the fact that comonad cohomology is functorial in its coefficients (recall the discussion after Definition~\ref{definition:BarResolution}).
\begin{Cor}\label{cor:functorial}
	Given a tensor functor $F\colon\Cat\to\Dat$ such that the functor $Z_F$ exists, then Davydov-Yetter cohomology defines a functor 
	$$
	H^n_{DY}(F,?,!)\colon\;\mathcal{Z}(F)^{op}\times \mathcal{Z}(F)\to \mathrm{Vec}_k, \qquad \text{for all} \; n\geq0.
	$$
\end{Cor}
 This corollary  can be used to compare cohomologies for different coefficients by using morphisms between them.

\subsection{Proof of Theorem \ref{theorem:Main}}
The proof consists of a sequence of lemmas. We need first to relate Davydov-Yetter cohomology to 
the complex from Proposition~\ref{proposition:MonadComonad}
associated to the central monad $Z_F$. 
This is guided by the following sketch presented for $F$ the identity functor and trivial coefficients:
\begin{align}
\Nat(\otimes^n, \otimes^n)&\cong \int_{X_1,...,X_n} \Hom_{\Cat}(X_1\otimes...\otimes X_n,X_1\otimes...\otimes X_n)\label{eq:31}\\
&\cong\int_{X_1,...,X_n} \Hom_{\Cat}(X_n^{\vee}\otimes...\otimes X_1^{\vee}\otimes X_1\otimes...\otimes X_n,I)\label{eq:32}\\
&\cong\Hom_{\Cat}\left(\int^{X_1,...,X_n}X_n^{\vee}\otimes...\otimes X_1^{\vee}\otimes X_1\otimes...\otimes X_n,I\right)\label{eq:33}\\
&\cong\Hom_{\Cat}(Z^n(I),I),\label{eq:34}
\end{align}
for $n>0$, while $n=0$ case is trivial: the space of natural endotransformations of the functor $\otimes^{0}\colon k \mapsto I_{\Cat}$ is isomorphic to $\mathrm{End}(I_\Cat)$.
The isomorphism $\eqref{eq:31}$ is a special case of the well known fact that 
$$
\Nat(F,G)=\int_{X}\Hom(F(X),G(X))
$$
 (compare e.g.~\cite[Chap.~IX.5]{M}). We note  that~\eqref{eq:32} follows from the definition of right duals and \eqref{eq:33} follows from the fact that the $\Hom$-functor preserves limits. We thus get an isomorphism~\eqref{eq:34} between the cochain groups  from Theorem \ref{theorem:Main} for $F=\id$ and trivial coefficients.
 To show that this isomorphism is also an isomorphism of cochain complexes (for general $F$ and coefficients) is the main body of technical work in this section.
 
\begin{proof}[Proof of Theorem~\ref{theorem:Main}]
	We begin with a lemma which  is a reformulation of Davydov-Yetter cohomology similar to the composition of isomorphisms \eqref{eq:31} \& \eqref{eq:32}.
\begin{Lemma}\label{lemma:DinatDY}
Let $F\colon\Cat\to\Dat$ be a tensor functor between finite tensor categories for which the functor $Z_F$ is well-defined. Moreover, let  $(X,\rho^X), (Y,\rho^Y)\in\mathcal{Z}(F)$. Then, the Davydov-Yetter complex $F\colon\Cat\to\Dat$ with coefficients $(X,\rho_X)$ and $(Y,\rho_Y)$ is isomorphic to the following complex: the cochain groups are
\begin{equation}\label{eq:Dinat-cochain}
\mathrm{Dinat}\left((?^{\vee} \circ\otimes^n\circ F^{\times n})\otimes X\otimes (\otimes^n\circ F^{\times n}),Y\right)\ .
\end{equation}
For a dinatural transformation $\gamma$ from~\eqref{eq:Dinat-cochain},
	\begin{equation}
	\gamma_{X_1,...,X_n}\colon \;F(X_n)^{\vee}\otimes...\otimes F(X_1)^{\vee}\otimes X\otimes F(X_1)\otimes...\otimes F(X_n)\to Y\ ,
	\end{equation}
the differential is
{\small
	\begin{align}\label{eq:lemmadifferential}
	\tilde{\delta}^n(\gamma)_{X_0,...,X_n}&:=\gamma_{X_1,...,X_n}\circ \left(\id_{FX_n^{\vee}\otimes...\otimes FX_1^{\vee}}\otimes \ev_{FX_0}\otimes \id_{X\otimes FX_1\otimes...\otimes FX_n}\right)\circ\nonumber\\
	&\mbox{}\qquad\circ \left(\id_{FX_n^{\vee}\otimes...\otimes FX_0^{\vee}}\otimes \rho^X_{X_0}\otimes\id_{FX_1\otimes...\otimes FX_n}\right)\nonumber\\
	&+\sum_{i=1}^n(-1)^i\gamma_{X_0,...,X_{i-1}\otimes X_i,...,X_n}\nonumber\\
	&+(-1)^{n+1}\left(\ev_{FX_n}\otimes \id_Y\right)\circ\left(\id_{FX_n^{\vee}}\otimes\rho^Y_{X_n}\right)\circ
\left(\id_{FX_n^{\vee}}\otimes\gamma_{X_0,..,X_{n-1}}\otimes\id_{FX_n}\right)
	\end{align}
}
	\end{Lemma}
\begin{Remark}\label{rem:tilde-delta}
	Similar to Remark~\ref{remark:graphicalDY}, we can express the above differential graphically:
{\footnotesize
	\begin{multline}
	\tilde{\delta}^n(\gamma)_{X_0,...,X_n}=\;\;\;\;\ipic{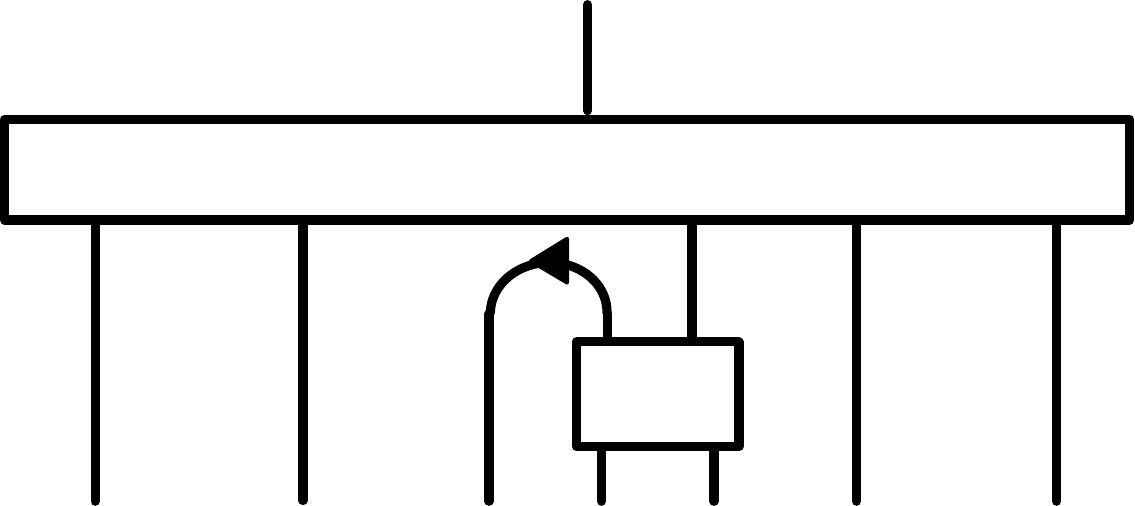}{0.3}\put(-75,-22){\tiny $\rho^X_{X_0}$}\put(-162,-45){\tiny $F(X_n)^{\vee}$}\put(-135,-45){\tiny $F(X_1)^{\vee}$}\put(-106,-45){\tiny $F(X_0)^{\vee}$}\put(-80,-45){\tiny $X$}\put(-68,-45){\tiny $F(X_0)$}\put(-46,-45){\tiny $F(X_1)$}\put(-18,-45){\tiny $F(X_n)$}\put(-83,10){\scriptsize $\gamma$}\put(-82,38){\tiny $Y$}\put(-140,-20){\tiny $\ldots$}\put(-30,-20){\tiny $\ldots$}\nonumber\\
	+ \; \sum_{i=1}^{n}(-1)^i\;\;\ipic{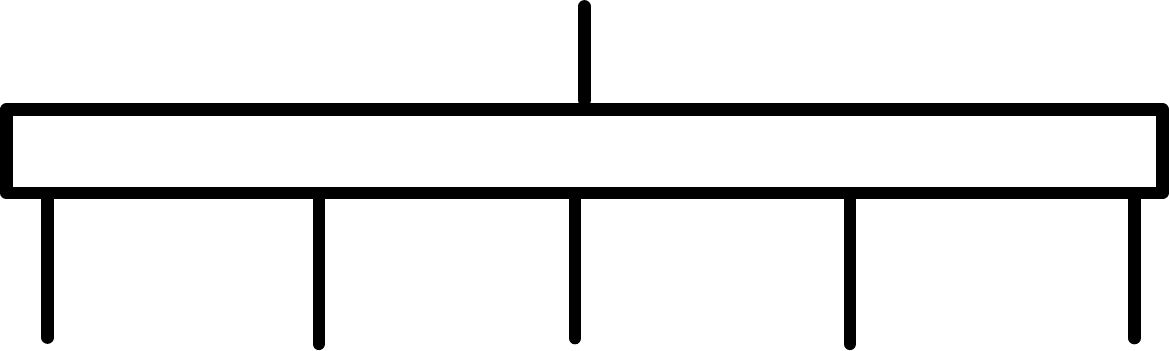}{0.3}\put(-175,-35){\tiny $F(X_n)^{\vee}$}\put(-145,-35){\tiny $F(X_{i-1}\otimes X_i)^{\vee}$}\put(-88,-35){\tiny $X$}\put(-68,-35){\tiny $F(X_{i-1}\otimes X_i)$}\put(-12,-35){\tiny $F(X_{n})$}\put(-88,28){\tiny $Y$}\put(-88,2){\scriptsize $\gamma$}\put(-145,-15){\tiny $\ldots$}\put(-110,-15){\tiny $\ldots$}\put(-70,-15){\tiny $\ldots$}\put(-30,-15){\tiny $\ldots$} 
	\;+\;(-1)^{n+1}\; \ipic{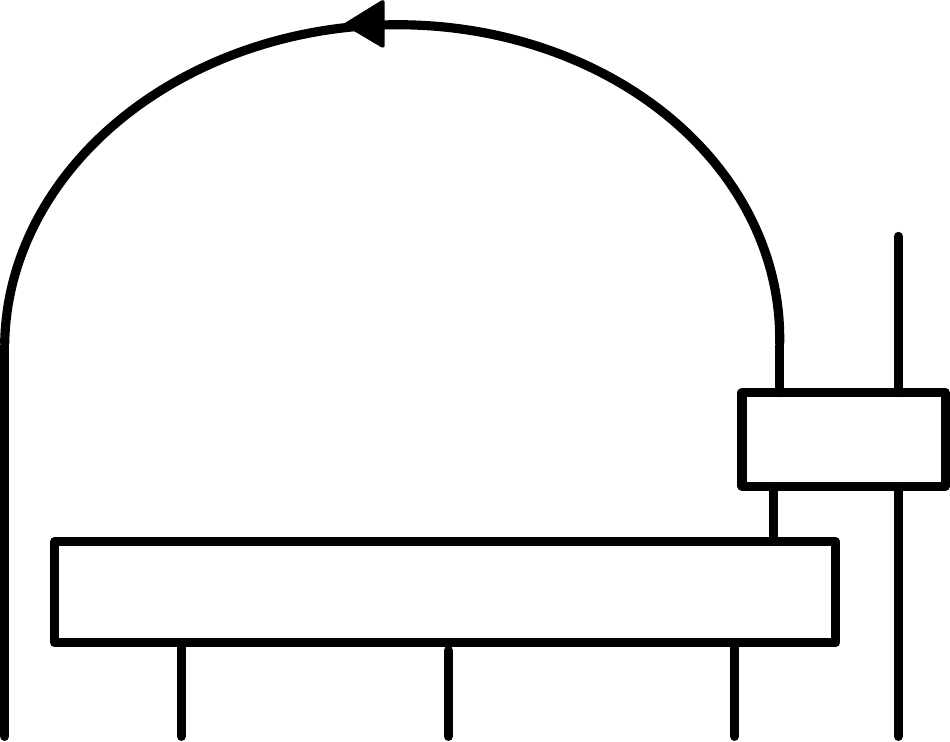}{0.3}\put(-155,-60){\tiny $F(X_n)^{\vee}$}\put(-123,-60){\tiny $F(X_{n-1})^{\vee}$}\put(-75,-60){\tiny $X$}\put(-15,-60){\tiny $F(X_{n})$}\put(-50,-60){\tiny $F(X_{n-1})$}\put(-75,-33){\scriptsize $\gamma$}\put(-20,-11.5){\tiny $\rho^Y_{X_n}$}\put(-10,22){\tiny $Y$}\put(-95,-50){\tiny $\ldots$}\put(-55,-50){\tiny $\ldots$}
	\end{multline}
	}
\mbox{}\!\!\!where we omit indices in $\gamma$ for brevity. We also note that for $n=0$ the cochain spaces are $\Hom_{\Dat}(X,Y)$ and in the differential $\tilde{\delta}^0$ above only first and last terms are present, and the coupon with $\gamma$ corresponds to a morphism from $X$ to $Y$, i.e.\ the sources $F(X_i)$ and $F(X_i)^{\vee}$, for $i\ne 0$, should be omitted in the picture of the differential. 
\end{Remark}
\begin{proof}[Proof of Lemma \ref{lemma:DinatDY}]
We first state the isomorphism of the cochain spaces. 
Using the graphical conventions introduced above,  
the isomorphism on the components of a natural transformation
 $f\in \Nat\left(X\otimes\left(\otimes^n \circ F^{\times n}\right),\left(\otimes^n\circ F^{\times n}\right)\otimes Y\right)$ is the following canonical map:
\be
\put(-30,-3){$\Psi:$}
\put(20,22){\tiny $F(X_1)$}   \put(45,15){\tiny $\ldots$} \put(55,22){\tiny $F(X_n)$}   \put(79,22){\tiny $Y$}
\put(45,-2){\scriptsize $f$}
\put(10,-25){\tiny $X$} \put(20,-25){\tiny $F(X_1)$}   \put(45,-15){\tiny $\ldots$} \put(55,-25){\tiny $F(X_n)$}  
\ipic{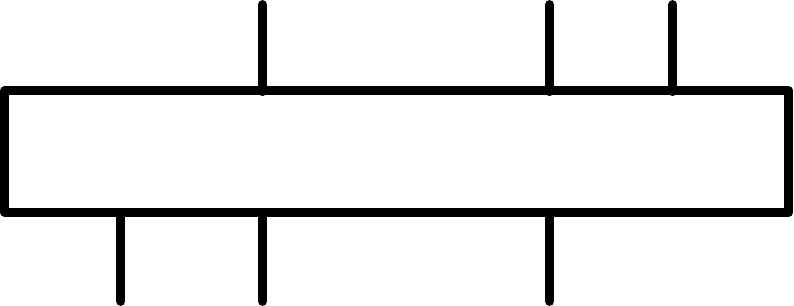}{.25}   \qquad  
\longmapsto
\qquad
\ipic{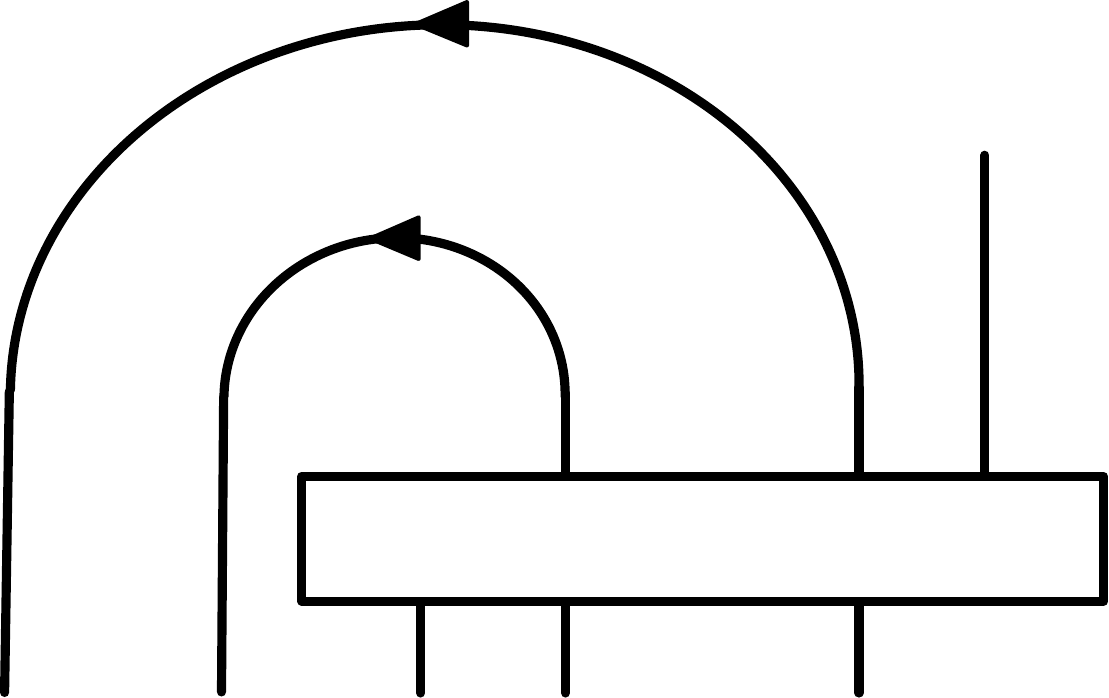}{.23}
\put(-50,-5){\tiny $\ldots$} \put(-16,25){\tiny $Y$}
\put(-50,-23){\scriptsize $f$}
\put(-115,-5){\tiny $\ldots$}
\put(-136,-45){\tiny $F(X_n)^{\!\vee}$}  \put(-106,-45){\tiny $F(X_1)^{\!\vee}$}   \put(-79,-45){\tiny $X$}   \put(-67,-45){\tiny $F(X_1)$}   \put(-50,-35){\tiny $\ldots$}   \put(-37,-45){\tiny $F(X_n)$}
\ee
where the dots indicate the evaluation on $\mathrm{ev}_{F(X_k)}\colon F(X_k)^{\vee}\otimes F(X_k)\to I$ for $2\leq k\leq n-1$. The inverse map $\Psi^{-1}$ is defined similarly using the coevaluation maps.

Using these maps, one can easily transport the differential via $\tilde{\delta}=\Psi\circ\delta\circ \Psi^{-1}$ and obtain~\eqref{eq:lemmadifferential}.
 We write $\delta^n=\sum_{i=0}^{n+1}(-1)^i\delta^n_i$ and show this for  $\delta^n_0$.
 The transported differential is on components
\begin{equation}
\ipic{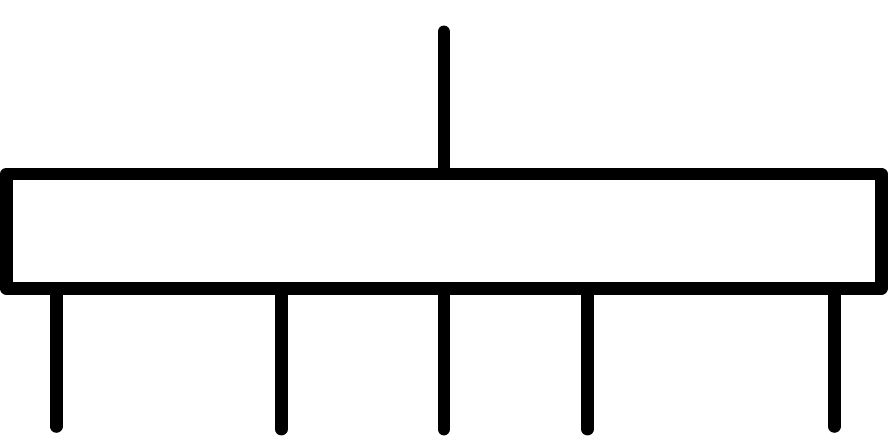}{0.25}\put(-115,-35){\tiny $F(X_n)^{\vee}$}\put(-85,-35){\tiny $F(X_1)^{\vee}$}\put(-57,-35){\tiny $X$}\put(-45,-35){\tiny $F(X_1)$} \put(-18,-35){\tiny $F(X_n)$}\put(-57,27){\tiny $Y$}\put(-57,-2.5){\scriptsize $\gamma$}\put(-90,-20){\tiny $\ldots$}\put(-25,-20){\tiny $\ldots$}
\qquad \stackrel{\Psi^{-1}}{\longmapsto} \qquad
 \ipic{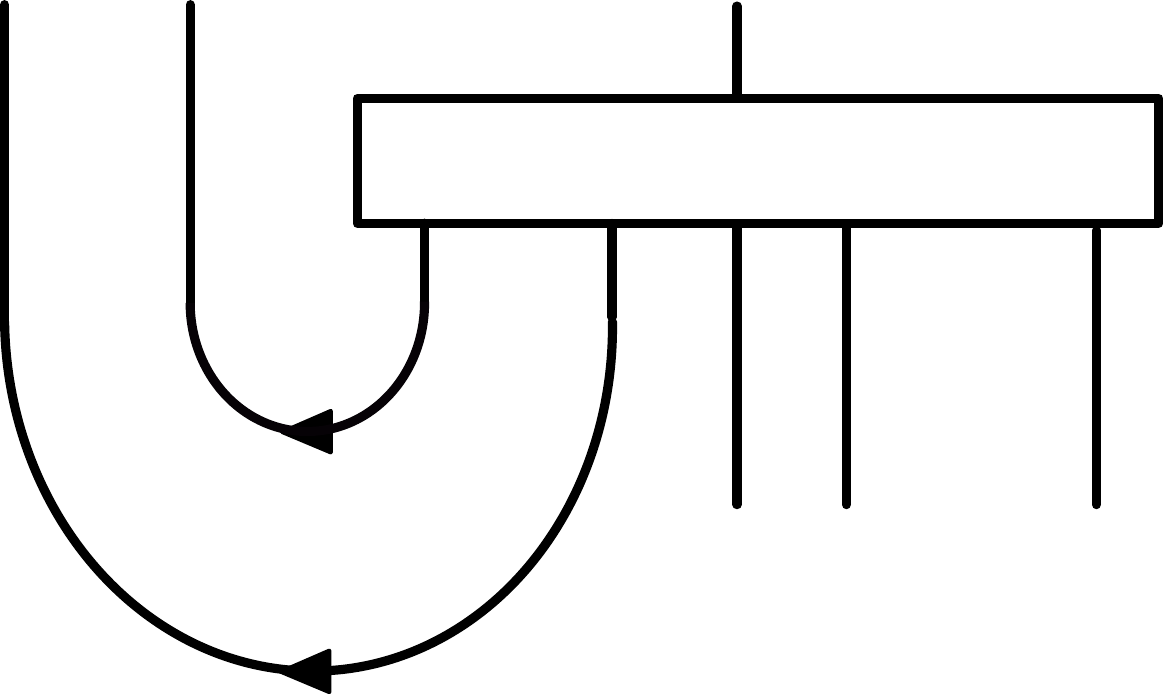}{0.25}\put(-150,44){\tiny $F(X_1)$} \put(-125,44){\tiny $F(X_n)$}\put(-54,44){\tiny $Y$}\put(-56,-27){\tiny $X$}\put(-45,-27){\tiny $F(X_1)$}\put(-17,-27){\tiny $F(X_n)$}\put(-53,20){\scriptsize $\gamma$}\put(-134,6){\tiny $\ldots$}\put(-83,6){\tiny $\ldots$}\put(-30,6){\tiny $\ldots$}\qquad \stackrel{\delta^n_0}{\longmapsto}\qquad
\end{equation}

\medskip

\begin{equation}
\ipic{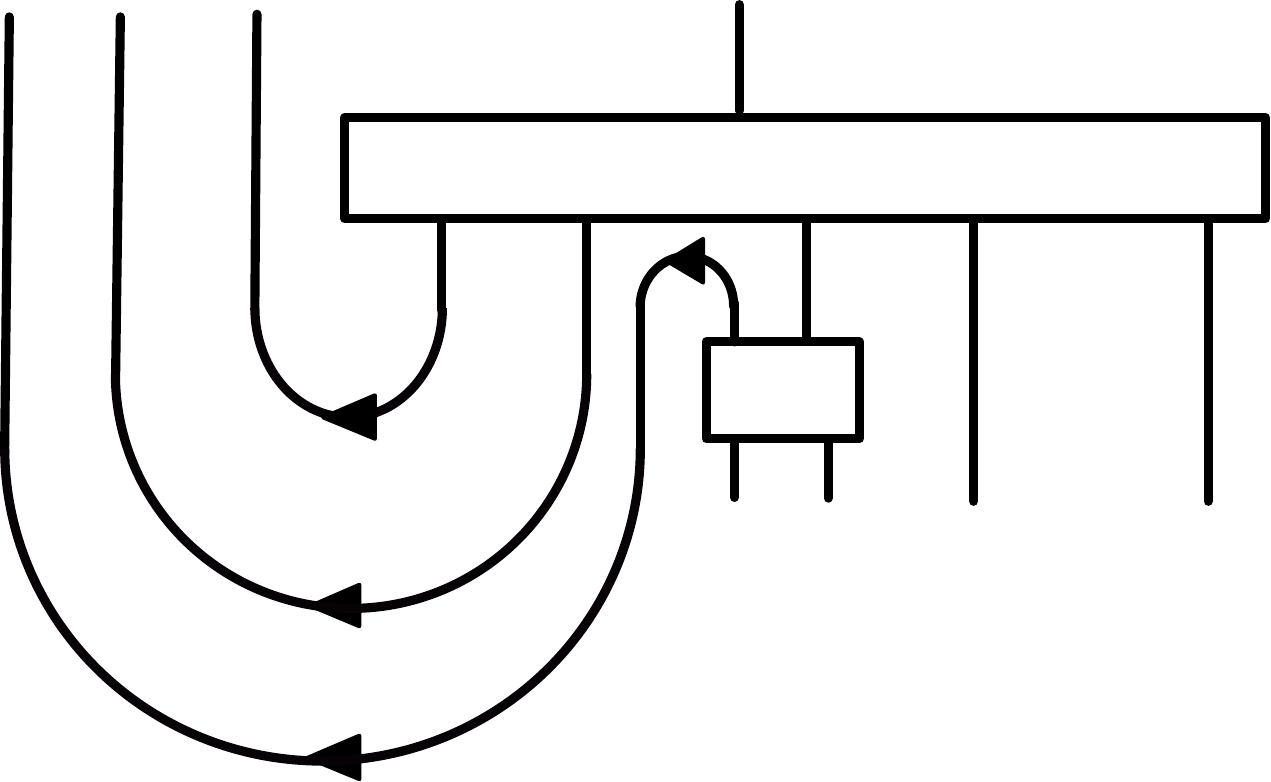}{0.3}\put(-197,58){\tiny $F(X_0)$}\put(-175,58){\tiny $F(X_1)$}\put(-153,58){\tiny $F(X_n)$}\put(-80,58){\tiny $Y$}\put(-84,-25){\tiny $X$}\put(-73,-25){\tiny $F(X_0)$}\put(-49,-25){\tiny $F(X_n)$}\put(-18,-25){\tiny $F(X_1)$}\put(-80,30){\scriptsize $\gamma$}\put(-75,-2.5){\tiny $\rho^X_{X_0}$}\put(-161,12){\tiny $\ldots$}\put(-115,12){\tiny $\ldots$}\put(-30,12){\tiny $\ldots$}
\qquad \stackrel{\Psi}{\longmapsto}\qquad
 \ipic{coupon-dinat4}{0.3}\put(-75,-22){\tiny $\rho^X_{X_0}$}\put(-162,-45){\tiny $F(X_n)^{\vee}$}\put(-135,-45){\tiny $F(X_1)^{\vee}$}\put(-106,-45){\tiny $F(X_0)^{\vee}$}\put(-80,-45){\tiny $X$}\put(-68,-45){\tiny $F(X_0)$}\put(-46,-45){\tiny $F(X_1)$}\put(-18,-45){\tiny $F(X_n)$}\put(-83,10){\scriptsize $\gamma$}\put(-80,40){\tiny $Y$}\put(-140,-20){\tiny $\ldots$}\put(-30,-20){\tiny $\ldots$}
\end{equation}

\medskip
The other summands can be computed similarly.
\end{proof}
We can now construct a canonical isomorphism between the complex from Lemma \ref{lemma:DinatDY} and the spaces $\Hom_{\Dat}\left(Z_F^n(X),Y\right)$, which corresponds to isomorphism \eqref{eq:33} in the outline.
\begin{Lemma}\label{lemma:DYMonad}
The complex presented in Lemma~\ref{lemma:DinatDY} is isomorphic to the complex with cochain vector spaces $\Hom_{\Dat}\bigl(Z_F^n(X),Y\bigr)$ and the differential 
\begin{align}\label{eq:Zcpx}
\partial^n(f)&:=f\circ Z_F^n\left(\beta_X\right)+\sum_{i=1}^n(-1)^i f\circ Z_F^{n-i}
\bigl(\mu^F_{Z_F^{i-1}(X)}\bigr)
+(-1)^{n+1}\beta_Y\circ Z_F(f),
\end{align}
where $\beta_X$ and $\beta_Y$ are defined as in \eqref{equation:Module} corresponding to $\rho^X$ and $\rho^Y$ respectively.
\end{Lemma}
\begin{proof}
We first define isomorphisms to the cochain groups \eqref{eq:Dinat-cochain}
 of the complex described in Lemma~\ref{lemma:DinatDY}. 
 Recall that $i(X)\colon F(?)^{\vee}\otimes X\otimes F(?)\to Z_F(X)$ denotes the universal dinatural transformations for the coend $Z_F(X)$. 
Let $i^{(n)}(X)$ denotes the universal dinatural transformation for the coend $Z^n_F(X)$, recall~\eqref{eq:i2} for $n=2$.  Given 
$$
\gamma\in\mathrm{Dinat}\left((?^{\vee}\circ\otimes^n\circ  F^{\times n})\otimes X\otimes (\otimes^n\circ F^{\times n}),Y\right),
$$
 we define $\hat{\gamma}\colon Z_F^n(X)\to Y$ as the unique morphism that makes the following diagram commute
\begin{equation}\label{equation:isomorphismZ}
\footnotesize
\begin{tikzpicture}
\matrix (m) [matrix of math nodes,row sep=3.5em,column sep=7em,minimum width=2em]
{
	F(X_n)^{\vee}\otimes...\otimes F(X_1)^{\vee}\otimes X\otimes F(X_1)\otimes...\otimes F(X_n)& Z_F^n(X)\\
 Y&  \\
};
\path[-stealth]

(m-1-1)  edge node [above] {\small $i^{(n)}_{X_1,...,X_n}(X)$} (m-1-2)
(m-1-1) edge node [left] {\small $\gamma_{X_1,...,X_n}$} (m-2-1)
(m-1-2) edge[dashed] node [below right] {\small $!\exists\;\hat{\gamma}$} (m-2-1)
;
\end{tikzpicture}
\end{equation}
The inverse map can be written down explicitly. Given a morphism $f\colon Z_F^n(X)\to Y$, we define the corresponding element $\tilde{f}$ from~\eqref{eq:Dinat-cochain}
component-wise via 
\be\label{eq:tildef}
\tilde{f}_{X_1,...,X_n}:=f\circ i^{(n)}_{X_1,...,X_n}(X).
\ee

 We write the differential in~\eqref{eq:Zcpx} as $\partial^n=\sum_{i=0}^{n+1}(-1)^i\partial^n_i$ and describe how the isomorphism $f\mapsto \tilde{f}$
  transports the corresponding summands of the differential from Lemma~\ref{lemma:DinatDY}. We begin with $\partial^n_0$.
For $n=0$ and $f\in~\Hom_{\Dat}(X,Y)$, we have the equality:
\be\label{eq:tilde-delta-0}
\tilde{\delta}^0_{0}(f) = f \circ (\ev_{FX_0}\otimes \id_X) \circ (\id_{FX_0^{\vee}}\otimes \rho_{X_0}^X) =  f\circ \beta_X \circ i_{X_0}(X),
\ee
where we used~\eqref{equation:Module}, recall also Remark~\ref{rem:tilde-delta}. The right hand side of~\eqref{eq:tilde-delta-0} factors uniquely through the coend $Z_F(X)$ and defines the map  $\partial^0_0= f\circ \beta_X\colon Z_F(X) \to Y$.
We similarly treat the $n>0$ cases. Let now $f\in~\Hom_{\Dat}(Z^n_F(X),Y)$, then the unique $\partial^n_{0}(f)$ is fixed by the following commuting diagram:
\begin{equation*}
\footnotesize
\begin{tikzpicture}
\matrix (m) [matrix of math nodes,row sep=5em,column sep=5.5em,minimum width=2em]
{
	F(X_n)^{\vee}\otimes...\otimes F(X_0)^{\vee}\otimes X\otimes F(X_0)\otimes...\otimes F(X_n)& & Z_F^{n+1}(X)\\
	F(X_n)^{\vee}\otimes...\otimes F(X_0)^{\vee} \otimes F(X_0)\otimes X\otimes...\otimes F(X_n)& Z^n_F(X) &\\
	F(X_n)^{\vee}\otimes...\otimes  X\otimes...\otimes F(X_n)&  &\\
	Y& &  \\
};
\path[-stealth]

(m-1-1)  edge node [left] {$\id\otimes\rho^X_{X_0}\otimes \id$} (m-2-1)
(m-2-1)  edge node [left] {$\id\otimes\ev_{F(X_0)}\otimes \id$} (m-3-1)
(m-3-1)  edge node [left] {$\tilde{f}_{X_1,...,X_n}$} (m-4-1)
(m-1-1)  edge node [above] {$i^{(n+1)}_{X_0,...,X_n}(X)$} (m-1-3)
(m-3-1)  edge node [above,sloped] {$i^{(n)}_{X_1,...,X_n}(X)$} (m-2-2)
(m-1-3)  edge node [above, sloped] {$Z_F^n\left(\beta_X\right)$} (m-2-2)
(m-2-2)  edge node [sloped, above left] {$f$} (m-4-1)
(m-1-3)  edge[bend left, dashed] node [below right] {$\partial^n_{0}(f)$} (m-4-1)

;
\end{tikzpicture}
\end{equation*}
The vertical composition is just $\tilde{\delta}^n_{0}(\tilde{f})$.
 The above diagram consists of an upper pentagon and a lower left triangle. The upper pentagon is simply the definition 
  of $Z^n_F(\beta_X)$,
  recall~\eqref{equation:Module},
   while the lower left triangle is the definition of $\tilde{f}$ in terms of $f$,
see~\eqref{eq:tildef}.
 Since both diagrams commute, the entire diagram commutes too. 
Comparing this diagram with the diagram in~\eqref{equation:isomorphismZ}, 
where $\gamma$ is the vertical composition $\tilde{\delta}^n_0(\tilde{f})$, it fixes $\partial_0^n(f)$ uniquely as the first term in~\eqref{eq:Zcpx}. 
 
 For $n>0$, the maps $\partial^n_i(f)$ for $0<i<n+1$ are computed via the following commuting diagram:
\begin{equation*}
\footnotesize
\begin{tikzpicture}
\matrix (m) [matrix of math nodes,row sep=6.5em,column sep=6em,minimum width=2em]
{
	F(X_n)^{\vee}\otimes...\otimes F(X_0)^{\vee}\otimes X\otimes F(X_0)\otimes...\otimes F(X_n)& & Z_F^{n+1}(X)\\
	 & Z_F^n(X) &\\
	Y& &\\
};
\path[-stealth]
(m-1-1)  edge node [above] {$i^{(n+1)}_{X_0,...,X_n}(X)$} (m-1-3)
(m-1-1)  edge node [midway, sloped,below] {$i^{(n)}_{X_0,...,X_{i-1}\otimes X_i,...,X_n}(X)$} (m-2-2)
(m-1-1)  edge node [left] {$\tilde{f}_{X_0,...,X_{i-1}\otimes X_i,...,X_n}$} (m-3-1)
(m-2-2)  edge node [sloped, above left] {$f$} (m-3-1)
(m-1-3)  edge node [sloped, above] {$Z_F^{n-i}\left(\mu^F_{Z_F^{i-1}(X)}\right)$} (m-2-2)
(m-1-3)  edge[bend left, dashed] node [below right] {$\partial^n_{i}(f)$} (m-3-1);
\end{tikzpicture}
\end{equation*}
Here, 
the upper triangle follows from the definition of the multiplication~\eqref{eq:muF-def} of the monad $Z_F$, while the lower left triangle is the definition of $\tilde{f}$ from~\eqref{eq:tildef}.
Comparing the above commuting diagram to~\eqref{equation:isomorphismZ} where $\gamma=\tilde{f}$, it fixes the map $\partial^n_i(f)$ uniquely as those in the sum in~\eqref{eq:Zcpx}.

Finally, we find for 
$n\geq0$ the term $\partial^n_{n+1}(f)$ is computed via the commuting diagram
\begin{equation*}
\footnotesize
\begin{tikzpicture}
\matrix (m) [matrix of math nodes,row sep=4.5em,column sep=4em,minimum width=2em]
{
	F(X_n)^{\vee}\otimes...\otimes F(X_0)^{\vee}\otimes X\otimes F(X_0)\otimes...\otimes F(X_n)& & Z_F^{n+1}(X)\\
	F(X_n)^{\vee}\otimes Y\otimes F(X_n)& Z_F(Y) &\\
	F(X_n)^{\vee}\otimes F(X_n)\otimes Y& &\\
	Y& &  \\
};
\path[-stealth]

(m-1-1)  edge node [left] {$\id_{F(X_n)^{\vee}}\otimes \tilde{f}\otimes \id_{F(X_n)}$} (m-2-1)
(m-2-1)  edge node [left] {$\id_{F(X_n)^{\vee}}\otimes\rho^Y_{X_n}$} (m-3-1)
(m-3-1)  edge node [left] {$\ev_{F(X_n)}\otimes\id_Y$} (m-4-1)
(m-1-1)  edge node [above] {$i^{(n+1)}_{X_0,...,X_n}(X)$} (m-1-3)
(m-2-1)  edge node [above] {$i_{X_n}(Y)$} (m-2-2)
(m-2-2)  edge node [sloped, above left] {$\beta_Y$} (m-4-1)
(m-1-3)  edge node [above, sloped] {$Z_F(f)$} (m-2-2)
(m-1-3)  edge[bend left, dashed] node [below right] {$\partial^n_{n+1}(f)$} (m-4-1)
;
\end{tikzpicture}
\end{equation*}
This  works analogous to the first diagram for $\partial^n_{0}$:
the upper triangle is by definition of~$Z_F(f)$, while the lower one is  by  definition~\eqref{equation:Module} of $\beta_Y$.
\end{proof}
We conclude the proof of Theorem~\ref{theorem:Main} by observing that the differential $\partial$ obtained in Lemma~\ref{lemma:DYMonad} is precisely of the form required in Proposition~\ref{proposition:MonadComonad}.
\end{proof}
\begin{Remark}
For the special case of trivial coefficients and $F=\id_{\Cat}$, a reformulation of Davydov-Yetter cohomology as a `Hochschild cohomology in tensor categories' is stated in \cite[Prop.~7.22.7]{EGNO}. The algebra in question is the `canonical algebra' $A$ in the tensor category $\Cat\boxtimes \Cat^{op}$, where $\boxtimes$ is the Deligne product, and it can be written as $A=\int^{X\in \Cat} X^{\vee}\boxtimes X$, see~\cite{Sh}. Therefore, due to Lemma~\ref{lemma:DinatDY} the Hochschild complex for $A$ is isomorphic to the complex introduced in Lemma~\ref{lemma:DYMonad} for $F=\id_{\Cat}$ and $\mathsf{X}=\mathsf{Y}=\mathsf{I}$.
\end{Remark}

\subsection{Ocneanu Rigidity}\label{ssec:Ocneanu}
An immediate application of Theorem \ref{theorem:Main} is a conceptual proof of Ocneanu rigidity. In this subsection we assume additionally that the field $k$ is of characteristic $0$ and algebraically closed.
Ocneanu rigidity in the sense that $H^n_{DY}(F)=0$ for a tensor functor $F$ between fusion categories and for all $n>0$ is proven in \cite[Sec.~7]{ENO}, using
 semisimple weak Hopf algebras. It is based on the construction of a homotopy contraction for the complex defining Davydov-Yetter cohomology, which makes crucial use of a left integral $\mu$ of the weak Hopf algebra such that $\mu(1)\neq 0$. The proof does not hold for \textsl{non-semisimple} finite tensor categories, including the case of weak Hopf algebra. The reason for this is that Maschke's theorem implies the absence of such left integrals for non-semisimple (weak) Hopf algebras. As will be shown in Section~\ref{ssec:example}, there are indeed examples of non-semisimple finite tensor categories with non-trivial Davydov-Yetter cohomology.

\begin{Lemma}\label{lemma:semisimpleZ-mod}
	Let $F\colon\Cat\to\Dat$ be a tensor functor between semisimple finite tensor categories. Then $Z_F\mathrm{-mod}$ is a semisimple finite tensor category.
\end{Lemma}
\begin{proof}
That   $Z_F\mathrm{-mod}$ is a finite $k$-linear category	 was proven in 
	\cite[Thm.~3.3]{Maj2} and \cite[Thm.~3.4]{Sh2}. It also follows from the discussion in \cite[Sec.~3.3]{Sh2} that $Z_F\mathrm{-mod}$ has a canonical structure of a tensor category. 
	
	 To show that  $Z_F\mathrm{-mod}$ is semisimple we  use Maschke's theorem for Hopf monads~\cite[Thm.~6.5~\&~Rem.~6.2]{BV1}. For a given Hopf monad $T$, the theorem states that the category $T\mathrm{-mod}$ is semisimple if and only if  $T$ admits a \textit{normalized cointegral}. We recall that a cointegral for a bimonad $T$ is  a morphism $\Lambda\colon I \to T(I)$ such that
	 \be
	 \mu_I \circ T(\Lambda) = \Lambda\circ \alpha,
	 \ee
	 where $\alpha\colon T(I)\to I$  is the structural map of the bimonad $T$,
recall the discussion above~\eqref{eq:alpha}.
	 	  A cointegral of $T$ is called \textit{normalized} if 
	  \be
	  \alpha\circ\Lambda=\id_{I}.
	  \ee
	  In our case of the Hopf monad $T=Z_F$ on $\Dat$, a normalized cointegral will be denoted by  
	  $$
	  \Lambda^F\colon I_\Dat\to Z_F(I_\Dat)
	  $$ 
	  and it should satisfy (if exists)
	\begin{equation}\label{eq:norm-coint-F}
	 \mu_{I_{\Dat}}^F\circ Z_F(\Lambda^F)=\Lambda^F\circ\alpha^F	\qquad \text{and}\qquad \alpha^F\circ\Lambda^F=\id_{I_\Dat},
	 \end{equation} 
	 where $\alpha^F$ is the structural map of $Z_F$ from~\eqref{eq:alpha}.
	 Therefore, to prove semisimplicity of $Z_F\mathrm{-mod}$ it is enough to show existence of  such a normalized cointegral $\Lambda^F$.

	 We first recall that the Drinfeld center $\mathcal{Z}(\Cat)$ of a fusion category $\Cat$ over an algebraically closed field of characteristic $0$ is semisimple,
	  see e.g.\ \cite[Thm.~9.3.2]{EGNO}, and is equivalent to $Z_F\mathrm{-mod}$ for $F=\id$.
	 Therefore, by Maschke's theorem, the central Hopf monad $Z$ admits a normalized cointegral $\Lambda:=\Lambda^{\id}$ satisfying~\eqref{eq:norm-coint-F} for $F=\id$.

	  We claim that 
	 \begin{equation}
	 \Lambda^F:=F(\Lambda)
	 \end{equation}
	  is a normalized cointegral for $Z_F$  for any tensor functor $F\colon\Cat\to\Dat$ between fusion categories. Indeed, we have that $F$ is exact as it is an additive functor between semisimple categories and therefore $F$ preserves colimits. Coends are a special case of colimits, 
	  and therefore for the coends $Z_F(V)$ in~\eqref{eq:ZF-def} we can choose  
	  $$
	  Z_F\bigl(F(M)\bigr) :=  F\bigl(Z(M)\bigr), \qquad M\in\Cat,
	  $$
	  and  for the corresponding dinatural transformations~\eqref{iF-def} 
  $$
  i_{X}^F\bigl(F(M)\bigr):=F\bigl(i_X(M)\bigr), \qquad X,M\in\Cat.
  $$ 
	   With this choice and the fact that $F$ is a strict tensor functor, we  obtain for the corresponding bimonad structure on $Z_F$:
	 $$
	 \mu^F_{I_{\Dat}}=F(\mu_{I_{\Cat}}) \qquad 
	  \text{and}\qquad \eta^F_{I_{\Dat}}=F(\eta_{\Cat})
	 $$ 
	 and
	 $$
	 \Psi^F_{FV,FW}=F(\Psi_{V,W})\qquad\text{and}\qquad \alpha^F=F(\alpha).
	 $$
	  Recall their definitions in~\eqref{eq:muF-def},~\eqref{eq:def-eta},~\eqref{eq:def_comultiplication} and~\eqref{eq:alpha}, correspondingly.
	    Moreover, we have $Z_F(\Lambda^F)=F(Z(\Lambda))$. 
	    
Recall now that~\eqref{eq:norm-coint-F} holds for $F=\id$, then we have
	  \begin{equation}
	  \mu_{I_\Dat}^F\circ Z_F(\Lambda^F)
	  =F\bigl(\mu_{I_\Cat}\circ Z(\Lambda)\bigr)\stackrel{\eqref{eq:norm-coint-F}}{=}F(\Lambda\circ\alpha)=\Lambda^F\circ\alpha^F
	  \end{equation}
	  and similarly
	  \begin{equation}
	  \alpha^F\circ\Lambda^F=F(\alpha\circ \Lambda)=F(\id_{I_{\Cat}})=\id_{I_{\Dat}}.
	  \end{equation}
	 We have thus shown that $F(\Lambda)$ is a normalized cointegral of $Z_F$, as claimed above, and therefore  $Z_F\mathrm{-mod}$ is semisimple by Maschke's theorem  for Hopf comonads.
\end{proof}
As a corollary, we can now use the relation to comonad cohomology in Theorem~\ref{theorem:Main} to obtain a new proof of the following generalization of Ocneanu rigidity.
\begin{Cor}[Ocneanu rigidity with coefficients]\label{cor:ocneanu}
	Let $F\colon\Cat\to\Dat$ be a tensor functor between semisimple finite tensor categories. Then, $H_{DY}^n(F,\mathsf{X},\mathsf{Y})=0$ for all $n>0$ and for all $\mathsf{X},\mathsf{Y}
\in \mathcal{Z}(F)$. In particular, we have $H^n_{DY}(F)=0$ for all $n>0$.
\end{Cor}
\begin{proof}
Every additive functor between semisimple categories is exact. Thus, the monad $Z_F$ exists and by  Theorem~\ref{theorem:Main} we can formulate Davydov-Yetter cohomology of $F$ as the comonad cohomology associated to $G_F$.
By Proposition~\ref{theorem:G-acyclicity},  the comonad cohomology of a $G_F$-projective object is $0$. It thus suffices to prove that any coefficient $\mathsf{X}$ in $Z_F\mathrm{-mod}$ is $G_F$-projective. 

The right adjoint in $\F_F\dashv \U_F$ is the forgetful functor and therefore faithful. Hence by Lemma~\ref{lemma:projectiveisGprojective}, every projective object in $Z_F\mathrm{-mod}$ is $G_F$-projective as well.  However, all  objects  in   $Z_F\mathrm{-mod}$  are projective, because $Z_F\mathrm{-mod}$ is semisimple by Lemma~\ref{lemma:semisimpleZ-mod}. 
\end{proof}
\begin{Remark}
 Lemma~\ref{lemma:semisimpleZ-mod} and thus Corollary~\ref{cor:ocneanu} remain true for any algebraically closed field $k$ in the case that $\dim \Cat\neq 0$.	
	 This is indeed	the case where the Drinfeld center $\mathcal{Z}(\Cat)$ of a fusion category $\Cat$ remains semisimple (compare with the proof of~\cite[Thm.~9.3.2]{EGNO}).
\end{Remark}

\section{Finite dimensional Hopf algebras}\label{sec:RepresentationCategories}
In this section, we apply constructions and results obtained in the two previous sections to the case of Hopf algebras.

 We consider a finite dimensional Hopf algebra  
 $(H,\mu,1,\Delta,\varepsilon,S)$ over a field $k$, where $\mu$ denotes the algebra multiplication, $1$ is the unit in $H$, $\Delta$ is the comultiplication, $\varepsilon$ is the counit, and $S$ is the antipode. We will use Sweedler's notation for comultiplication: 
 $$
 \Delta(h)= h_{(1)}\otimes h_{(2)}.
 $$
 
By $\Hmod$ we denote the rigid category of finite dimensional (left) modules over~$H$.
In Subsection~\ref{ssec:CentralH-mod}, we  describe the central monad $Z$ and the corresponding comonad $G$ for the case $\Cat=\Hmod$ and $F=\id$, together with the bar resolution and the corresponding Davydov-Yetter complex. 
 In Subsection~\ref{ssec:G-proj-Hopf},  we discuss the notion of $G$-projective modules and relate them to $H^*$ projectiveness.
 In Subsection~\ref{ssec:forgetful}, we study the Davydov-Yetter complex of the forgetful functor and reformulate it as Davydov-Yetter complex of the identity functor with a non-trivial coefficient. 

Let us introduce the following $H$-modules:
\begin{itemize}
	\item The \emph{trivial module} $_{\varepsilon}V$ associated to a vector space $V$. The action is $h.v=\varepsilon(h)v$ with $h\in H$ and $v\in V$.
	\item The \emph{regular module} $H_{\mathrm{reg}}$ is the vector space $H$ with the action being the left multiplication.
	\item The \emph{coregular module} $H^*_{\mathrm{coreg}}$ is the vector space $H^*$ with the action defined by 
	$$
	h.f=f(?h)\ .
	$$ 
	\item The \emph{coadjoint module} $H^*_{\mathrm{coad}}$ is $H^*$ as a vector space with the action 
	\be\label{eq:coad}
	h.f=f\bigl(S(h_{(1)})?h_{(2)})\bigr)\ .
	\ee
	\item The  module $\left(H^{*\otimes n}\otimes V\right)_{\mathrm{coad}}$, for any  $V\in\Hmod$ and $n\geq1$, with the action 
	\begin{equation}\label{eq:HV-coadj}
	h.(a_1\otimes...\otimes a_n \otimes v)=a_{1}\left(S(h_{(1)})?h_{(2n+1)}\right)\otimes\dots\otimes a_n\left(S(h_{(n)})?h_{(n+2)}\right)\otimes  h_{(n+1)}v,
	\end{equation}
	for $a_i\in H^*$, $1\leq i \leq n$, and   $v\in V$.
	 Notice that this module is in general not isomorphic to the $n$-fold tensor product of $H^*_{\mathrm{coad}}$ and $V$.
\end{itemize}
Furthermore, we note that the vector space $H^*$ admits a canonical Hopf-algebra structure with the unit $1_{H^*}:=\varepsilon$ and the multiplication $\mu_{H^*}$   defined by
\begin{equation}\label{equation:multiplication}
 \mu_{H^*}(f\otimes g)(h):=(f*g)(h):=f(h_{(2)})g(h_{(1)})
\end{equation}
 for $h\in H$,  the comultiplication is $\Delta_{H^*}=\mu^*$ and the counit is defined by $\varepsilon_{H^*}\colon f\mapsto f(1)$.

 \subsection{The central monad for $\Hmod$}\label{ssec:CentralH-mod}
Recall for this subsection the definition of the central monad $Z=Z_{\id}$ in Subsection \ref{ssec:CentralMonad}.
\begin{Proposition}\label{proposition:enocomplex}
	The central monad $Z$ on $\Hmod$ is given by the following data:
 \begin{itemize}
 \item As a functor, it sends $V$ to $(H^*\otimes V)_{\mathrm{coad}}$, i.e.\ $Z(V)=H^*\otimes_k V$ with $H$-action given by
	\begin{equation}
	h.(f\otimes v)=f\left(S\left(h_{(1)}\right)?h_{(3)}\right)\otimes h_{(2)}.v,
	\end{equation} 
	for $f\in H^*$, $h\in H$ and $v\in V$. It acts on a morphism $\psi\colon V\to W$ as $Z(\psi)=\id_{H^*}\otimes \psi$. 
	\item The multiplication $\mu_V\colon Z^2(V)\to Z(V)$ given by 
	\begin{equation}\label{eq:mu-Z}
	\mu_V(f\otimes g\otimes v)=(f*g)\otimes v,
	\end{equation}
	with $*$ defined in~\eqref{equation:multiplication}, $f,g\in H^*$ and $v\in V$.
	\item The unit $\eta_V\colon V\to Z(V)$ is given by $\eta_V(v)=\varepsilon\otimes v$.
	\end{itemize}
\end{Proposition}
\begin{proof}
  The universal dinatural transformation is defined on components via 
  \begin{align}
  &i_X\colon\; X^{\vee}\otimes V\otimes X\to H^{^*}\otimes V,\notag\\
  &i_X(f\otimes v\otimes x)=f(?.x)\otimes v,
  \end{align}
  for $f \in X^{\vee}, x\in X$ and $v\in V$.
   It was proven for the case $V=I$ in \cite[Sec.~3.3]{Ly} and \cite[Lem.~3]{K} that this indeed yields a dinatural transformation with the universal property. The general case can be checked analogously. For the multiplication and the unit it is straightforward to check that the defining equations~\eqref{eq:muF-def} and~\eqref{eq:def-eta} are satisfied.
\end{proof}

A statement analogous to Proposition~\ref{proposition:enocomplex} was made in~\cite[Ex.~3.12]{Sh3} for the central \textsl{comonad}. 

 In the Hopf algebra case, the Drinfeld center of $\Hmod$ is equivalent to the category of finite dimensional modules over the Drinfeld double $D(H)$.
  As a vector space, the Drinfeld double\footnote{Our conventions here coincide with those of~\cite[Sec.\,7]{Maj3}.}
   of a finite dimensional Hopf algebra $H$ is 
\begin{equation}
D(H):=H^*\otimes_k H.
\end{equation}
This vector space admits an algebra structure with unit $1_{H^*}\otimes 1$ and multiplication such that $H^*\otimes 1$ and $1_{H^*}\otimes H$ are subalgebras identified with $(H^*,*)$ and $(H,\cdot)$, respectively, and 
\begin{equation}\label{eq:drinfeldmultiplication}
\psi\cdot h := \psi \otimes h\ , \quad
h \cdot \psi :=\psi\bigl(S(h_{(1)})? h_{(3)}\bigr)\otimes h_{(2)}\ , \qquad h\in H,\ \psi \in H^*\ ,
\end{equation}
where we identify $\psi\in H^*$ with $\psi\otimes 1$ and $h\in H$ with $1_{H^*}\otimes h$.

The following Proposition follows from~\cite{DS}.
\begin{Proposition}\label{proposition:DZ-isomorphism}
The categories $D(H)\mathrm{-mod}$ and $Z\mathrm{-mod}$ are isomorphic.
More precisely, an object $(V,\beta)\in Z\mathrm{-mod}$ corresponds to the unique $D(H)$-module with the underlying space $V$ and the following action:
\begin{equation}\label{eq:Z-DH-act-2}
(\psi\otimes h).v=\beta(\psi\otimes h.v),\qquad \psi\in H^*,\; h\in H,\; v\in V.
\end{equation}
where  $h.v$ denotes the $H$-action on $V$.\\
And conversely, a $D(H)$-module $V$ corresponds to the underlying $H$-module with the structure of $Z$-module $\beta\colon H^* \otimes V \to V$ given by the action of the subalgebra $H^*\subset D(H)$ on~$V$.
\end{Proposition}
\begin{proof}
	We check that the action in~\eqref{eq:Z-DH-act-2} is indeed a $D(H)$ action. Recall the relations~\eqref{eq:drinfeldmultiplication}.
For $\psi\in H^*$ and $h\in H$, we have
	\begin{equation}
	\psi.(h.v)= \beta(\psi\otimes h.v) = (\psi\otimes h).v =(\psi\cdot h).v \ ,
\ee
and
	\begin{align}
	h.(\psi.v)&=h.\beta(\psi\otimes v)=\beta\left(\psi\left(S\left(h_{(1)}?h_{(3)}\right)\right)\otimes h_{(2)}.v\right)\nonumber\\
	&=\left(\psi\left(S\left(h_{(1)}?h_{(3)}\right)\right)\otimes h_{(2)}\right).v=(h\cdot\psi).v
	\end{align}
	by the fact that $\beta\colon Z(V)\to V$ is an $H$-module homomorphism.
	Finally, we have for $\psi,\phi\in H^*$:
	\begin{equation}
	\psi.(\phi.v)=\beta(\psi\otimes\phi.v)=\beta(\psi\otimes \beta(\phi\otimes v))\stackrel{\dagger}{=}\beta(\psi*\phi\otimes v)=(\psi*\phi).v,
	\end{equation}
	where $\dagger$ is due to  commutativity of the left diagram in~\eqref{eq:Tmod} (for $T=Z$) and we also used~\eqref{eq:mu-Z}.
\end{proof}

We recall that $D(H)\mathrm{-mod}$ is monoidally equivalent to the Drinfeld center $\mathcal{Z}(\Hmod)$. Then the isomorphism in Proposition~\ref{proposition:DZ-isomorphism} is 
 a corollary of Proposition~\ref{theorem:equivalencecentralizer} for $F=\id$.

We can now reformulate  Davydov--Yetter complex for $\Hmod$ with coefficients using Lemma~\ref{lemma:DYMonad}. 
Recall that for an $H$-module $X$ we have $Z^n(X)=\left(H^{*\otimes n}\otimes X\right)_{\mathrm{coad}}$.
\begin{Cor}\label{cor:DYH-mod}
Given $D(H)$-modules $X$ and $Y$, the Davydov--Yetter complex of $\Hmod$ with coefficients in $X$ and $Y$  is
	\begin{equation}
	 C^n_{DY}(\Hmod,X,Y) \cong \Hom_{H}\bigl((H^{*\otimes n}\otimes X)_{\mathrm{coad}}, Y\bigr)
	 \end{equation}
	  with the differential 
	\begin{align}
\partial^n(f)
(a_0\otimes\dots\otimes a_n\otimes x)=& 
a_0. f(a_1\otimes\dots\otimes a_n\otimes x)\nonumber\\
	&+\sum_{i=1}^{n} (-1)^i f\bigl(a_0\otimes\dots\otimes (a_{i-1}* a_{i})\otimes\dots\otimes a_n\otimes x\bigr)\nonumber\\
	&+(-1)^{n+1}f(a_0\otimes\dots\otimes a_{n-1}\otimes a_n. x),
	\end{align}
	with $a_0,a_1,\dots,a_n\in H^*$ and $x\in X$.
\end{Cor}
\begin{Remark}\label{remark:signisomorphis}
	The differential
$\partial^n$ in Corollary~\ref{cor:DYH-mod} is $(-1)^{n+1}$ times
	the differential
$\partial^n$
 in Lemma~\ref{lemma:DYMonad}.
	The two complexes are isomorphic via the following isomorphism: The $n$th cochains are multiplied by a sign, which is $+1$ if $n$ is $1$ or $2$ modulo $4$ and $-1$ otherwise.
\end{Remark}
 \begin{Remark}
The complex from Corollary~\ref{cor:DYH-mod} with trivial coefficients is (up to an isomorphism) the complex that was introduced in \cite[Sec.\,6]{ENO} for weak Hopf algebras in order to prove Ocneanu rigidity.
\end{Remark}

Recall the comonad $G:=G_{\id}$ defined  in~\eqref{eq:G_T}  with the counit $\varepsilon$ in~\eqref{G-eps-def} for $T=Z$.
We have for $(V,\beta)\in Z\mathrm{-mod}$  and $n\geq1$
\be\label{eq:Gn-Hopf}
G^{n}\colon\; (V,\beta) \mapsto
\left((H^{*\otimes n}\otimes V)_{\mathrm{coad}},\,
\mu_{H^*}\otimes \id_{H^*}^{\otimes (n-1)}\otimes \id_{V}\right).
\ee
where the  $H$-module $(H^{*\otimes n}\otimes V)_{\mathrm{coad}}$ is defined in~\eqref{eq:HV-coadj}.
Notice that from  coassociativity of the coproduct we have
\begin{equation}
(H^*\otimes (H^*\otimes V)_{\mathrm{coad}})_{\mathrm{coad}}=(H^{*\otimes 2}\otimes V)_{\mathrm{coad}}.
\end{equation}

We note that using  the isomorphism in Proposition~\ref{proposition:DZ-isomorphism}, the   $H$-module $(H^{*\otimes n}\otimes V)_{\mathrm{coad}}$ in~\eqref{eq:Gn-Hopf} has also a $D(H)$-action where $H^*$ acts via $\mu_{H^*}\otimes \id_{H^{*\otimes (n-1)}\otimes V}$.
We now rewrite the bar resolution~\eqref{eq:barresolution} of $G$ using this action.
\begin{Cor}
For $X\in D(H)\mathrm{-mod}$, the bar resolution of $X$ associated to $G$ is a complex in $D(H)\mathrm{-mod}$ of  the form
\be\label{eq:G-DH-cpx}
\ldots\xrightarrow{d_n} (H^{*\otimes n}\otimes X)_{\mathrm{coad}}\xrightarrow{d_{n-1}}\ldots\xrightarrow{d_{2}}(H^{*\otimes 2}\otimes X)_{\mathrm{coad}}\xrightarrow{d_1}(H^{*}\otimes X)_{\mathrm{coad}}\xrightarrow{\beta} X\to 0
\ee
with 
$$
d_n= \id^{\otimes n}_{H^*}
\otimes \beta+\sum_{i=1}^n (-1)^{i}~\id_{H^*}^{\otimes (n-i)}\otimes \mu_{H^*}\otimes\id_{H^*}^{\otimes (i-1)}\otimes \id_X,
$$
and $\beta$ is the action of $H^*\subset D(H)$ on $X$. For the trivial $D(H)$-module, $\beta$ is given by $\varepsilon_{H^*}$.
\end{Cor}

\newcommand{\Ind}{\mathrm{Ind}}
\subsection{$G$-projective modules as induced modules}\label{ssec:G-proj-Hopf}
By Theorem~\ref{theorem:Main}, we can compute  Davydov--Yetter cohomologies using the bar resolution~\eqref{eq:G-DH-cpx}, or any other $G$-resolution. 
The $G$-resolutions  are made of $G$-projective modules -- a certain class of modules over $D(H)$. 
Here, we discuss what $G$-projectivity means in the case of Hopf algebras.
Due to Proposition~\ref{proposition:DZ-isomorphism}, we will often identify objects from $Z\mathrm{-mod}$ with those from $D(H)\mathrm{-mod}$.
We have thus to describe $G$-projective objects in terms of $D(H)$-modules.

We have the canonical embedding of Hopf algebras $H\to D(H)$, and have thus the induction functor
\be\label{eq:Ind-V}
\Ind\colon\; V\, \mapsto\, \Ind_H^{D(H)} V := D(H)\otimes_H V\ .
\ee
We note that as the vector space $\Ind(V)$ is $H^*\otimes_k V$: indeed, $D(H)$ is $H^*\otimes_k H$ as a vector space and thus the $H$ tensorand  goes through the balanced tensor product over $H$ in~\eqref{eq:Ind-V} and acts on $V$. The image of this action is $V$ of course.  
We then recall that the $D(H)$ action on $\Ind(V)$ is defined via multiplication:
\be
\rho_{\Ind(V)}\colon \; D(H)\otimes  D(H)\otimes_H V \xrightarrow{\,\mu_{D(H)}\otimes{\id_V}\,} D(H)\otimes_H V.
\ee
Let $\psi\otimes v\in H^*\otimes_k V$ and $\phi\in H^*$, then the $H^*$-action on $\Ind(V)=H^*\otimes_k V$ 
is given just by multiplication on the left:
\be\label{eq:Hstar-act}
\phi . (\psi\otimes v) = (\phi * \psi)\otimes v\ ,
\ee
while the $H$-action on $\Ind(V)$ is (recall the multiplication in~\eqref{eq:drinfeldmultiplication})
\be
h . (\psi\otimes v) = (h \cdot \psi)\otimes v\  = \psi\bigl(S(h_{(1)})? h_{(3)}\bigr)\otimes h_{(2)}.v.
\ee
Compare this $D(H)$ action with the action in the tuple $\mathcal{F}(V)=(Z(V),\mu_V)$, where $\mathcal{F}$ denotes the free functor associated to the monad $Z$ defined in Proposition~\ref{proposition:enocomplex}.
We obtain the following:

\begin{Proposition}\label{prop:FreeVsInd}
	 $\Ind(V)$ and $\mathcal{F}(V)$ are isomorphic as $D(H)$-modules.
\end{Proposition}

And we get the following corollary (recall also Corollary~\ref{cor:G-proj}):

\begin{Cor}\label{cor:DH-G-proj} A $D(H)$-module is $G$-projective if and only if it is a direct summand of the induced module $\Ind(V)$ for some $V\in\Hmod$.
\end{Cor}
\begin{proof}
	We need to show that $\mathcal{F}(V)$ is $G$-projective.
	We show that $\mathcal{F}(V)$ is a retract of $G(\mathcal{F}(V))$.
	Notice that $G(\mathcal{F}(V))=\mathcal{F}\circ \mathcal{U}\circ\mathcal{F}(V)$.
	We define 
	\begin{equation}
	s=\mathcal{F}(\eta_V):\mathcal{F}(V)\to G(\mathcal{F}(V)),
	\end{equation}
		where $\eta$ is the unit of the monad $Z$.
	Notice that $G(\mathcal{F}(V))=(Z^2(V),\mu_{Z(V)})$ and $\mu_{V}\colon Z^2(V)\to Z(V) $ is a surjective morphism in $Z-\mod$.
	We find for the corresponding morphisms in $H-\mod$:
	\begin{equation}
	\mu_{V}\circ s=\mu_{V}\circ \mathcal{F}(\eta_V)=\mu_{V}\circ Z(\eta_V)=\id_{\mathcal{F}(V)}.
	\end{equation}
	Therefore, by Proposition~\ref{prop:FreeVsInd}, $\Ind(V)$ is isomorphic to a direct summand of a $G$-projective object and, by Corollary~\ref{cor:G-proj}, it is $G$-projective and so are all its direct summands.
	The converse direction is obvious.
\end{proof}

Recall that we have the canonical embedding of algebras $H^*\to D(H)$.
We then note from~\eqref{eq:Hstar-act} that the  $H^*$-module $\Ind(V)|_{H^*}$ 
  is isomorphic to the direct sum $(H^*)^{\oplus \dim(V)}$, where $H^*$ is   the regular representation space of $H^*$. We thus conclude with the following corollary:

\begin{Cor} \label{cor:G-proj-H}
$G$-projective modules are projective as $H^*$-modules. 
\end{Cor}

We note that $G$-projective modules are not necessarily projective as $H$-modules. An important class of such  $G$-projective modules appears in our example section~\ref{ssec:example} in constructing $G$-resolutions: as $H$-modules they are direct sums of one-dimensional modules, and in particular, they are non-projective as $D(H)$-modules.

\subsection{Cohomology of the forgetful functor}\label{ssec:forgetful}

The representation category $\Hmod$ comes with a canonical fiber functor: the forgetful functor 
\begin{equation}
\mathcal{U}_H\colon H\mathrm{-mod}\to\mathrm{Vec}_k.
\end{equation}
It is well known that the Davydov-Yetter cohomology of the forgetful functor is isomorphic to the Hochschild cohomology of the algebra $(H^*,*)$ with the trivial bimodule coefficient (see e.g. \cite[Prop.~7.4]{ENO}).	
In this subsection, we reformulate Hochschild cohomology of $(H^*,*)$ in a different direction: 
It is isomorphic to Davydov-Yetter cohomology of the identity functor with a non-trivial  coefficient. 
The following diagram displays the relations between complexes made precise in this section:
 \begin{equation*}
 \begin{tikzpicture}
 \matrix (m) [matrix of math nodes,row sep=3em,column sep=10em,minimum width=2em]
 {
 \footnotesize{\text{DY of the forgetful functor}}& 	\footnotesize{\text{Hochschild of } (H^*,*)} \\
 \footnotesize{\text{DY of } \id\; \text{with a  coefficient}} & 
\footnotesize{\text{Comonad of}\;\id\;\text{with a coefficient}}\\
 };
 \path[-stealth]
 
 (m-1-1)  edge node [above] {$\small{\cite[\text{Prop.~7.4}]{ENO}}$} (m-1-2)
 (m-1-2) edge node [right] {\small $\text{Theorem~\ref{cor:HochschildasCoefficients}}$} (m-2-2)
 (m-1-1) edge [dashed] node [left, sloped] {$$} (m-2-1)
 (m-2-2) edge node [below] {$\small{\text{Theorem~\ref{theorem:Main}}}$} (m-2-1)
 ;
 \end{tikzpicture}
 \end{equation*}
 where all arrows indicate isomorphisms of cochain complexes.
 We first explain what we mean by \textit{non-trivial coefficient}.
  For the $H$-module $H^*_{\mathrm{coreg}}$, we define the following map:
 	\begin{equation}\label{eq:beta}
 \beta_c\colon Z\left(H^*_{\mathrm{coreg}}\right)\to H^*_{\mathrm{coreg}},\;\;\; \beta_c (f\otimes g)(h):=f\left(S\left(h_{(1)}\right)h_{(3)}\right)g(h_{(2)})
 \end{equation}
  	for $f,g\in H^*$ and $h\in H$.
  	 \begin{Lemma}\label{lemma:beta}
  	 	The linear map $\beta_c$ from~\eqref{eq:beta} equips $H^*_{\mathrm{coreg}}$ with the structure of a $Z$-module.
  	 \end{Lemma}
  	 \begin{proof}
	 		We  first check that $\beta_c$ defines an $H$-module homomorphism. 
	 		For $a\in H$
  	 	\begin{align}
  	 	\beta_c(a.(f\otimes g))(h)&=f\left(S\left(a_{(1)}\right)S\left(h_{(1)}\right)h_{(3)}a_{(3)}\right)g\left(h_{(2)}a_{(2)}\right)\nonumber\\
  	 	&=f\left(S\left((ha)_{(1)}\right)(ha)_{(3)}\right)g\left((ha)_{(2)}\right)=a.\beta_c(f\otimes g)(h).
  	 	\end{align}
  	 	We then directly verify the axioms~\eqref{eq:Tmod} for a $Z$-action.  The right diagram in~\eqref{eq:Tmod} is
  	 	\begin{equation}
  	 	(\beta_c\circ\eta(f))(h)=\varepsilon\left(S\left(h_{(1)}\right)h_{(3)}\right)f(h_{(2)})=f\left(\varepsilon\left(h_{(1)}\right)h_{(2)}\varepsilon \left(h_{(3)}\right)\right)=f(h).
  	 	\end{equation}
    	 	We now check the left diagram of \eqref{eq:Tmod}  by calculating both directions in the diagram. For $p,q,f\in H^*$, we have
	 	\begin{align}
  	 	\beta_c\circ \mu_{H^*_{\mathrm{coreg}}}(p\otimes q\otimes f)=&\beta_c(\mu_I(p\otimes q)\otimes f)(h)\nonumber\\
  	 	=& (p* q)\left(S\left(h_{(1)}\right)h_{(3)}\right)f\left(h_{(2)}\right)\nonumber\\
  	 	\stackrel{\ddagger}{=}& p\left(S\left(h_{(1)}\right)h_{(5)}\right)q\left(S\left(h_{(2)}\right)h_{(4)}\right)f\left(h_{(3)}\right)
  	 	\end{align}
  	 	where $\ddagger$ follows from
  	 	\begin{align}
  	 	\Delta\otimes \id\left(S\left(h_{(1)}\right)h_{(3)}\otimes h_{(2)}\right)=& S\left(h_{(1)}\right)_{(1)}h_{(3)(1)}\otimes S(h_{(1)})_{(2)}h_{(3)(2)}\otimes h_{(2)}\nonumber\\
  	 	=& S\left(h_{(2)}\right)h_{(4)}\otimes S\left(h_{(1)}\right)h_{(5)}\otimes h_{(3)}.
  	 	\end{align}
  	 	The other direction is
  	 	\begin{align}
  	 	\beta_c\circ Z(\beta_c)(p\otimes q\otimes f)(h)=&\beta_c(p\otimes \beta_c(q\otimes f))(h)\nonumber\\
  	 	=& p\left(S(h_{(1)})h_{(3)}\right)\beta_c(q\otimes f)\left(h_{(2)}\right)\nonumber\\
  	 	=& p\left(S\left(h_{(1)}\right)h_{(5)}\right)q\left(S\left(h_{(2)}\right)h_{(4)}\right)f\left(h_{(3)}\right).
  	 	\end{align}
  	 	As both directions coincide the diagram commutes. 
		  	 	This completes the proof.
  	 \end{proof}
 \begin{Theorem}\label{cor:HochschildasCoefficients}
 	The Hochschild cochain complex $C^{\bullet}_{\mathrm{HH}}(H^*,k)$ is isomorphic to the comonad complex $C^{\bullet}\left((I,\alpha),\Hom_{\mathcal{Z}(\Hmod)}(?,(H^*_{\mathrm{coreg}},\beta_c))\right)_G$.
 \end{Theorem}
As an immediate corollary of this theorem, using Theorem~\ref{theorem:Main} we get that the Davydov-Yetter complex of the forgetful functor $C^{\bullet}_{DY}(\mathcal{U}_H)$ is isomorphic to the Davydov-Yetter complex 
 of the identity functor with a non-trivial coefficient: $C^{\bullet}_{DY}(\id, \mathsf{I},\mathsf{H})$, where $\mathsf{H}=(H^*_{\mathrm{coreg}},\rho_c)$ and $\rho_c$ denotes the image of $\beta_c$ under the isomorphism explained in~\eqref{eq:iso-braiding}.
 
Before proving Theorem~\ref{cor:HochschildasCoefficients}, we first prove the following two lemmas.
\begin{Lemma}\label{lemma:forgetfulrightadjoint}
The forgetful functor $\mathcal{U}_H\colon\Hmod\to \mathrm{Vec}_k$ has a right adjoint from $\mathrm{Vec}_k$ to $\Hmod$, with action $V\mapsto H^*_{\mathrm{coreg}} \otimes {_\varepsilon V}$. In particular, there is a natural family of isomorphisms
\be\label{eq:isoU}
\Hom_{H}\left(X,H^*_{\mathrm{coreg}}\right)
\xrightarrow{\;\cong\;}
\Hom_k\left(\mathcal{U}_H(X),k\right),\qquad
 f\mapsto \bar{f}:=f(?)(1),
\ee
for $X\in \Hmod$.
\end{Lemma} 
\begin{proof}
It is straightforward to check that the inverse to the map in~\eqref{eq:isoU} is 
\be\label{eq:g-tilde}
g\mapsto \tilde{g},\qquad \tilde{g}(x)(h):=g(h.x),
\ee
 for $g\in \Hom_k(\mathcal{U}_H(X),k)$,  $h\in H$ and $x\in X$. 
Naturality in $X$ for the map~\eqref{eq:isoU} is easy to check.
 \end{proof}

With the identification of Corollary~\ref{cor:DYH-mod}, we can reformulate Davydov-Yetter complex of the forgetful functor on $\Hmod$ using Proposition~\ref{proposition:enocomplex}.
\begin{Lemma}\label{proposition:HochschildComonad1}
The Hochschild complex of the algebra $(H^*,*)$  with trivial coefficients is isomorphic to the complex with cochain groups $\Hom_{H}\left(\left(H^{*\otimes n}\right)_{\mathrm{coad}}, H^*_{\mathrm{coreg}}\right)$ and differential
\begin{align}\label{eq:deltap}
\delta'(g)(a_0\otimes\dots\otimes a_n)(h)= & a_0\left(S\left(h_{(1)}\right)h_{(3)}\right)g(a_1\otimes\dots\otimes a_n)(h_{(2)})\nonumber\\
& +\sum_{i=1}^n (-1)^i g(a_0\otimes\ldots\otimes a_{i-1}*a_i\otimes\dots\otimes a_n)(h)+\nonumber\\
& (-1)^{n+1}g(a_0\otimes\dots\otimes a_{n-1})(h)a_n(1),
\end{align}
where $g\in \Hom_{H}\left((H^{*\otimes n})_{\mathrm{coad}}, H^*_{\mathrm{coreg}}\right)$ and $h\in H$ and $a_i\in H^*$ for $0\leq i\leq n$.
\end{Lemma}
\begin{proof}
Recall that the Hochschild complex is $\Hom_k(H^{*\otimes n},k)$ with a differential $\delta= \sum_{i=0}^{n+1}(-1)^i\delta_i$ that acts on cochains $f$ via
\begin{align*}
 \delta_0(f)(a_0\otimes\dots\otimes a_n)&=a_0(1)f(a_1\otimes\dots\otimes a_n)\\
 \delta_i(f)(a_0\otimes\dots\otimes a_n)&= f(a_0\otimes \dots\otimes (a_{i-1}*a_i)\otimes \dots\otimes a_n)\\
 \delta_{n+1}(f)(a_0\otimes\dots\otimes a_n)&=f(a_0\otimes\dots\otimes a_{n-1})a_{n}(1).
 \end{align*}
We directly transport this differential along the isomorphisms in Lemma~\ref{lemma:forgetfulrightadjoint}.
Let  $g\in \Hom_{H}\left((H^{*\otimes n})_{\mathrm{coad}}, H^*_{\mathrm{coreg}}\right)$ and we recall the definition of $\bar{?}$ and $\tilde{?}$ notations  from~\eqref{eq:isoU} and~\eqref{eq:g-tilde}, then
\begin{align}
\delta'(g)(a_0\otimes\dots\otimes a_n)(h):=
&\widetilde{\delta(\bar{g})}(a_0\otimes\dots\otimes a_n)(h)\nonumber\\
=&\delta(\bar{g})(h.(a_0\otimes\dots\otimes a_n))\nonumber\\
=&\delta(\bar{g})\left(a_0\left(S\left(h_{(1)}\right)?h_{(2n+2)}\right)\otimes\dots\otimes a_n\left(S(h_{(n+1)})?h_{(n+2)}\right)\right)\nonumber\\
\stackrel{\dagger}{=}& a_0\left(S\left(h_{(1)}\right)h_{(3)}\right)\bar{g}\left(h_{(2)}.(a_1\otimes\dots\otimes a_n)\right)\nonumber\\
&+\sum_{i=1}^{n} (-1)^i \bar{g}\left(h.(a_0\otimes\dots\otimes (a_{i-1}* a_{i})\otimes\dots\otimes a_n)\right)\label{eq:delta-delta}\\
&+(-1)^{n+1}\bar{g}\left(h.(a_0\otimes\dots\otimes a_{n-1})\right)a_n
\left(1\right) \nonumber
\end{align}
which equals the right hand side of~\eqref{eq:deltap}.
 We show the equality $\dagger$  for the  $\delta_i$  summands of~$\delta$ for $0\leq i\leq n+1$. 
For the first term it is straightforward, for the last term corresponding to $\delta_{n+1}$ we use the antipode and counit axioms. 
For the terms corresponding to $\delta_i$  for $1\leq i\leq n$, without loss of generality  we show it for $i=1$:  for all $b\in H$  the argument of $\bar{g}$ in $\delta_1$ is simplified as
\begin{align}
& a_0\left(S\left(h_{(1)}\right)?h_{(2n+2)}\right)*a_1\left(S\left(h_{(2)}\right)? h_{(2n+1)}\right)(b)\otimes \dots\otimes a_n\left(S\left(h_{(n+1)}\right)?h_{(n+2)}\right)\nonumber\\
&= a_0\left(S\left(h_{(1)}\right)b_{(2)}h_{(2n+2)}\right)a_1\left(S(h_{(2)})b_{(1)} h_{(2n+1)}\right)\otimes \dots\otimes a_n\left(S\left(h_{(n+1)}\right)?h_{(n+2)}\right)\nonumber\\
&= a_0\left(S\left(h_{(1)}\right)_{(2)}b_{(2)}(h_{(2n)})_{(2)}\right)a_1\left(S\left(h_{(1)}\right)_{(1)}b_{(1)} (h_{(2n)})_{(1)}\right)\otimes \dots\otimes a_n\left(S\left(h_{(n)}\right)?h_{(n+1)}\right)\nonumber\\
&= a_0\left(\left(S\left(h_{(1)}\right)bh_{(2n)}\right)_{(2)}\right)a_1\left(\left(S\left(h_{(1)}\right)b h_{(2n)}\right)_{(1)}\right)\otimes \dots\otimes a_n\left(S\left(h_{(n)}\right)?h_{(n+1)}\right)\nonumber\\
&= a_0*a_{1}\left(S\left(h_{(1)}\right)bh_{(2n)}\right)\otimes \ldots\otimes a_n\left(S\left(h_{(n)}\right)?h_{(n+1)}\right),\label{eq:arg-d1}
\end{align}
where in the second equality we used that the antipode is a coalgebra anti-homomorphism. We thus see from~\eqref{eq:arg-d1} that the argument of $\bar{g}$ in $\delta_1$ is indeed  $h.(a_0*a_{1} \otimes a_2\otimes \ldots\otimes a_n)$ as in~\eqref{eq:delta-delta}.
For the other summands in $\delta$ the calculation is similar. This completes the proof.
\end{proof}
We can now put everything together to prove Theorem~\ref{cor:HochschildasCoefficients}.
\begin{proof}[Proof of Theorem~\ref{cor:HochschildasCoefficients}]
We observe that the cochain complex with the differential~\eqref{eq:deltap} can be written as
\begin{equation}
\delta'(g)=\beta_c\circ Z(g)+\sum_{i=1}^n(-1)^ig\circ Z^{i-1}\left(\mu_{Z^{n-i}(I)}\right)+(-1)^{n+1}g\circ Z^n(\alpha),
\end{equation}
 where  $\alpha\colon H^*_{\mathrm{coad}}\to k$ is the canonical $Z$-module action on $I$ defined by $\alpha(f)=f(1)$.
  This is isomorphic to the complex from Proposition~\ref{proposition:MonadComonad},
 setting $\beta_Y=\beta_c$ and $\beta_X=\alpha$, 
   via the isomorphism from Remark~\ref{remark:signisomorphis}, which completes the proof. 
	\end{proof}
 The Davydov-Yetter complex of the identity functor is contained in the Davydov-Yetter complex of the forgetful functor. This admits a simple expression in our reformulation.
 \begin{Remark}
 Let $i\colon I\to H^*_{\mathrm{coreg}}$ be the canonical embedding of $I$ defined by $i\colon 1\mapsto \varepsilon$.
 It is straightforward to check that it  induces a $Z$-module map from $(I,\alpha)$ to $(H^*_{\mathrm{coreg}},\beta_c)$. 
 Therefore, by Corollary~\ref{cor:functorial} we have a map from $H^n_{DY}(\Hmod)$ to $H^n_{DY}(\mathcal{U}_H)$, which is just the map induced by the map of the corresponding cochain complexes.
\end{Remark}

\section{Example: the Hopf algebras $\Lambda \mathbb{C}^k\rtimes \C[\Z_2]$}\label{ssec:example}
The exterior algebras $\Lambda \C^k$ are Hopf algebras in the symmetric category $\mathrm{SVec}_{\C}$ of complex super vector spaces. Hence, their $2^{k+1}$-dimensional `bosonizations' $B_k:=\Lambda\C^k\rtimes \C[\Z_2]$ are Hopf algebras in the usual sense, i.e.\ in the category of complex vector spaces. Compare e.g. \cite{AEG}. As an algebra they are generated by one group-like generator $g$ and $k$ skew-primitive generators $x_1,\dots ,x_k$ being subject to the relations
\begin{align}
gx_i=-x_i g,\;\;\;x_i^2=0,\;\;\; x_ix_j=-x_jx_i,\;\;\; g^2=1,
\end{align}
with $1\leq i,j\leq k$. This becomes a Hopf algebra with the following coalgebra structure and antipode
\begin{align}
\Delta(g)=& g\otimes g, \;\;\;\Delta(x_i)=1\otimes x_i+x_i\otimes g,\nonumber\\
\varepsilon(g)=& 1,\;\;\;\varepsilon(x_i)=0,\nonumber\\
S(g)=& g,\;\;\;\;S(x_i)=gx_i.
\end{align} 
The first member of this family, $B_1$, is also known as \emph{Sweedler's 4-dimensional Hopf algebra}.

In this section we will prove the following theorem.

 \begin{Theorem}\label{proposition:DYSweedler}
 For the dimensions of the Davydov-Yetter cohomologies of the identity and forgetful functor on the representation categories $B_k\mathrm{-mod}$ we have
 	\begin{equation}\label{eq:SweedlerDY}
 	\dim H_{DY}^n(B_k\mathrm{-mod})=\dim H^n_{DY}(\mathcal{U}_{B_k})=\begin{cases}
 	0&\textrm{for n odd},\\
 	{k+n-1\choose n}& \text{for n even}.
 	\end{cases}
 	\end{equation}
 \end{Theorem}
  \begin{Remark}
  	In particular, we have $H_{DY}^3(B_k\mathrm{-mod})=0$. Hence, $B_k\mathrm{-mod}$ does not admit non-trivial first order deformations. Nevertheless, $\dim H^2_{DY}(B_k\mathrm{-mod})=\frac{(k+1)k}{2}$, which implies the existence of non-trivial first order deformations of the identity functor, which are furthermore unobstructed. 
	We give few explicit examples in Remark~\ref{rem:DY-exmp}. 
Already the case of $B_1$ shows that Ocneanu rigidity does not hold for general non-semisimple finite tensor categories.
  \end{Remark}
  
The proof is based on our reformulation of DY cohomologies (Theorem~\ref{theorem:Main}) and the representation theory of the Drinfeld double $D(B_k)$. More precisely, we  construct a (non-trivial) $G$-resolution for the tensor unit $(I,\alpha)$ in $D(B_k)\mathrm{-mod}$ and then apply the functor
 $\Hom_{Z\mathrm{-mod}}(?,(I,\alpha))$, recall the isomorphism of categories in Proposition~\ref{proposition:DZ-isomorphism}. By Theorem~\ref{theorem:Comparison}, the resulting complex is quasi-isomorphic to the comonad $G$ complex with trivial coefficients in Theorem~\ref{theorem:Main}, and hence to the Davydov-Yetter complex of the identity functor.
We can use the same $G$-resolution in the case of the forgetful functor, but here we apply the coefficient functor $\Hom_{Z\mathrm{-mod}}(?,(H^*_{\mathrm{coreg}},\beta_c))$.
 
Let 
\be\label{eq:epm}
e_{\pm}:=\frac{1\pm g}{2}
\ee
 denote the idempotents of the algebra $B_k$.
 The following are indecomposable modules over $B_k$ that we will make use of:
\begin{itemize}
	\item The two one-dimensional simple modules $I_{\pm}$ with one generator $v$ such that $x_i.v=0$ and $g.v=\pm v$. We denote the trivial module with $I=I_+$ as well.
	\item The projective covers $P_{\pm}$ of $I_{\pm}$, they are given by
	$$
	P_\pm := B_k \cdot e_{\pm}
	$$
	 and they are $2^{k}$-dimensional.
\end{itemize}

The Drinfeld double $D(B_k)$ has additional generators (those of the subalgebra $B_k^*$)
	\begin{equation}
	y_i:=x_i^*-(x_ig)^*\qquad\text{and}\qquad h:=1^*-g^*,
	\end{equation}
  where $?^*$ denotes the dual basis elements of the basis     in $B_k$:
   $$
   \{x_{i_1}\dots x_{i_l}g^r \, | \, 1\leq i_1<\dots<i_l\leq k, 0\leq l \leq k, r\in\mathbb{Z}_2\}\ .
   $$
  These generators are subject to the following relations, recall~\eqref{equation:multiplication} and~\eqref{eq:drinfeldmultiplication},
  \begin{equation}\label{eq:relationdD1}
  h^2=1,\qquad \{y_i,y_j\}=0,\qquad \{y_i,h\}=0
  \end{equation}
   and 
  \begin{equation}\label{eq:relationD2}
  [g,h]=0,\qquad \{y_i,g\}=0,\qquad \{x_i,h\}=0,\qquad
   \{x_i,y_j\}=\delta_{i,j}(1-hg)
  \end{equation}
  for all $1\leq i,j\leq k$ and with $\{a,b\}:=ab+ba$ denoting the anticommutator.
The last relation implies that on any $D(B_k)$-module the action of the generator $h$ is determined by the actions of the other generators, and therefore we will often suppress it in the discussion.
  
   The following are some indecomposable modules of $D(B_k)$ that we will make use of (compare \cite[Prop.~3.10 \& Sec.~3.7]{FGR}\footnote{We note that conventions on the Drinfeld double in~\cite{FGR} are slightly different but the two doubles are isomorphic.}):
	\begin{itemize}
		\item The two one-dimensional simple modules $\mathcal{I}_{\pm}:=\mathrm{Span}(v)$ with $x_i.v=y_i.v=0$ while the action of $g$ is given by $\pm1$. Note that the action of $h$ is then fixed by the relations~\eqref{eq:relationD2} to be  $g$. In particular, we have for the tensor unit 
		$$
		\mathcal{I}=\mathcal{I}_{+} = (I,\alpha).
		$$
		\item The projective (and injective) simple modules $\mathcal{A}_{\pm}$ of dimension $2^k$  are defined as
		\begin{equation}\label{eq:Apm-def}
		\mathcal{A}_{\pm}:=\mathrm{Span}\Big\{x_1^{i_1}\dots x_k^{i_k}v_{\pm}\,|\,(i_1,\dots,i_k)\in \mathbb{Z}_2^k\Big\},
		\end{equation}
		where $v_{\pm}$ is a cyclic vector such that $y_i.v_{\pm}=0$ and $g.v_{\pm}=\pm v_{\pm}$, and $h.v_{\pm}= \mp v_{\pm}$. We note that $\mathcal{A}_{\pm}$  considered as a $B_k$-module is isomorphic to $P_\pm$.

		\item The modules $\mathcal{B}_{\pm}$  are $P_{\pm}$ as $B_k$-modules and with the trivial action $y_i.v=0$ for all $v\in \mathcal{B}_{\pm}$. In this case, we have  that $h$ acts as $g$. We note that these modules are reducible but indecomposable.
		
		\item The modules $\mathcal{C}_{\pm}$:  let $f_\pm=\frac{1\pm h}{2}$ denote the primitive idempotents of $B_k^*$, then $\mathcal{C}_{\pm}$ as a  $B_k^*$-module is defined as
		\be\label{eq:Cpm-def}
		\mathcal{C}_{\pm} := B_k^* \cdot f_\pm
		\ee
		while the $B_k$ action is fixed via $x_i.v=0$ for all $v\in \mathcal{C}_{\pm}$ and $g$ acts  as $h$. These modules  are also reducible but indecomposable.
		
		\item We will use the notation $B^*_{k,\mathrm{coad}} = (B^*_{k})_\mathrm{coad}$ for the coadjoint module defined as in~\eqref{eq:coad}.
	\end{itemize}

We have the following simple lemma.

\begin{Lemma}\label{lem:D-simples}
The modules $\mathcal{A}_{\pm}$ and $\mathcal{I}_{\pm}$ exhaust all simple $D(B_k)$-modules up to isomorphism. Their isomorphism class is uniquely determined by the action of the pair $(g,h)$ on the cyclic vector: $(\pm,\mp)$ corresponds to $\mathcal{A}_{\pm}$ while $(\pm,\pm)$ corresponds to $\mathcal{I}_{\pm}$.
\end{Lemma}
		
In the following lemma we decompose the $G$-projective module $(Z(I),\mu_I)=\left(B^*_{k,\mathrm{coad}},\mu_{I}\right)$.
Direct summands of this module are $G$-projective
 and we will use them as building blocks for a $G$-resolution in Lemma~\ref{lemma:G-resolution}.
\begin{Lemma}\label{lemma:DecomCoad}
We have the following decomposition of $D(B_k)$-modules:
\begin{equation}\label{eq:GI-decomp-DB}
 G(\mathcal{I})=\left(B^*_{k,\mathrm{coad}},\mu_{I}\right)\cong\mathcal{A}_{(-)^k}\oplus \mathcal{C}_{+}
 \end{equation}
 and
 \begin{equation}
 G(\mathcal{I}_-)=(Z(I_-),\mu_{I_-})\cong\mathcal{A}_{(-)^{k+1}}\oplus \mathcal{C}_{-},
 \end{equation}
 where we set $(-)^{k}:=(-1)^k$.
\end{Lemma}
\begin{proof}
To prove the decomposition of  $G(\mathcal{I})$ in~\eqref{eq:GI-decomp-DB}, we first analyze the  $B_k$-action in the coadjoint representation.
On the basis elements in $B_k^*$, we have
$$
g.(x_{i_1}\dots x_{i_m})^*=(-1)^m(x_{i_1}\dots x_{i_m})^*\qquad g.(x_{i_1}\dots x_{i_m}g)^*=(-1)^m(x_{i_1}\dots x_{i_m}g)^*
$$
and 
\begin{align}
x_j.(x_{i_1}\dots x_{i_m})^*=&0 \qquad \text{for all}\qquad j,\label{eq:j-act-1}\\
x_j.(x_{i_1}\dots x_{i_m}g)^*=&\begin{cases}
2(-1)^{m-l+1}(x_{i_1}\dots \hat{x}_{i_l}\dots x_{i_m}g)^*&\textrm{for}\;\; i_l=j\\
0& \text{for}\;\;i_l\neq j\;\;\forall l,
\end{cases}\label{eq:j-act-2}
\end{align}
where the notation $\hat{x}_{i_l}$ means that we omit  the corresponding element.
From this action, we obtain the following $B_k$-submodules in a basis:
\begin{align}
B_k.(x_1\dots x_kg)^* & = \mathrm{Span} \left\{ (x_{i_1}\dots x_{i_m}g)^* | 1 \leq i_1 < i_2 < \dots < i_m \leq k \right\}\nonumber\\
& \cong P_{(-)^k}\label{eq:P-basis}
\end{align}
and
\be\label{eq:I-basis}
 \mathrm{Span} \left\{ (x_{i_1}\dots x_{i_m})^* | 1 \leq i_1 < i_2 < \dots < i_m \leq k \right\}  \cong I_+^{\oplus 2^{k-1}}\oplus I_{-}^{\oplus 2^{k-1}} .
\ee
We note that the isomorphism in~\eqref{eq:P-basis} is easy to establish after identifying the cyclic vector $w=(x_1 x_2 \dots x_k g)^*$, where $g$ acts by $(-1)^{k}$, with the cyclic vector $e_{(-)^k}$ of $P_{(-)^k}$ defined in~\eqref{eq:epm}. The isomorphism in~\eqref{eq:I-basis} is obvious.
We therefore have a decomposition over the subalgebra $B_k$:
\be\label{eq:GI-decomp}
 G(\mathcal{I})|_{B_k} = P_{(-)^k} \oplus I_+^{\oplus 2^{k-1}}\oplus I_{-}^{\oplus 2^{k-1}}.
\ee
Next, we compute  the actions of  $y_i\in B_k^*$.
Recall that $B_k^*$ acts via the multiplication on $B_k^*$ defined by $\phi*\psi=\phi\otimes\psi\circ \Delta^{op}$ for $\phi,\psi\in B_k^*$.
 We use the coproduct formula for the basis elements of $B_k$
\begin{align}
\Delta(x_{i_1}\ldots x_{i_m} g^r)
&=(1\otimes x_{i_1}+x_{i_1}\otimes g)\ldots (1\otimes x_{i_m}+x_{i_m}\otimes g)g^r\otimes g^r\\
&=\sum_{b\in\mathbb{Z}_2^{\times m}}x_{i_1}^{b_1}\ldots x_{i_m}^{b_m}g^r\otimes x_{i_1}^{1-b_1}g^{b_1}\dots x_{i_m}^{1-b_m}g^{b_m}g^r\label{eq:coproductformula},
\end{align}
where $r\in\mathbb{Z}_2$,
to calculate the products
\be\label{eq:y-act-0}
y_{i_l}.(x_{i_1}\dots {x}_{i_l}\dots x_{i_m})^*= y_{i_l}.(x_{i_1}\dots {x}_{i_l}\dots x_{i_m}g)^*  = 0 
\ee
and 
\be\label{eq:y-act}
\begin{split}
y_{i_l}.(x_{i_1}\dots \hat{x}_{i_l}\dots x_{i_m})^*=&(-1)^{m-l}(x_{i_1}\dots  x_{i_l} \dots x_{i_m})^*,\\
y_{i_l}.(x_{i_1}\dots\hat{x}_{i_l}\dots x_{i_m}g)^*=&(-1)^{m-l-1}(x_{i_1}\dots x_{i_l} \dots x_{i_m}g)^*.
\end{split}
\ee
With these explicit actions, we are now able to analyze the decomposition of $G(\mathcal{I})$ over $D(B_k)$. 
We first note that  $B_k^*$ acts on the direct summand $P_{(-)^k}$ in~\eqref{eq:GI-decomp} because of its basis given in~\eqref{eq:P-basis}. We claim that the resulting $D(B_k)$ module is isomorphic to $\mathcal{A}_{(-)^k}$, recall its definition in~\eqref{eq:Apm-def}. First, the resulting $D(B_k)$ module has the cyclic vector $w=(x_1 x_2 \dots x_k g)^*$ such that $y_i w = 0$ for  $1\leq i \leq k$. Secondly, this module has action of $h=-g$ and it is indecomposable. Due to the classification in Lemma~\ref{lem:D-simples}, this module should have $\mathcal{A}_{(-)^k}$ as a simple subquotient but they both have the same dimension. Therefore, the first direct summand in~\eqref{eq:GI-decomp} is indeed isomorphic to $\mathcal{A}_{(-)^k}$. For the reader's convenience we also present the $D(B_k)$ action schematically in the left part of Figure~\ref{fig:arrows}.

We now analyze the second part of~\eqref{eq:GI-decomp}. Again, from the $y_i$ actions in~\eqref{eq:y-act-0} and~\eqref{eq:y-act} the  summand $I_+^{\oplus 2^{k-1}}\oplus I_{-}^{\oplus 2^{k-1}}$ is closed under the action of $B_k^*$. It has a cyclic vector $1^*$ with the action $h.1^* = 1^*$. Moreover, the action of the subalgebra generated by $y_i$, $1\leq i \leq k$, is free as follows from~\eqref{eq:y-act}. We therefore have that the resulting $D(B_k)$-module is a projective module over $B_k^*$ isomorphic to $B_k^*. f_+$, i.e.\ we identify the cyclic vector $1^*$  with~$f_+$. Finally, the action of $x_i$'s is trivial due to~\eqref{eq:j-act-1}, and so the submodule is identified with~$\mathcal{C}_+$,
recall the definition in~\eqref{eq:Cpm-def} (see also  the right part of Figure~\ref{fig:arrows}). This concludes the proof of~\eqref{eq:GI-decomp-DB}.

\newcommand{\ssign}[1]{\text{\scriptsize$(-)^{#1}$}}
\begin{figure}[!ht]
{\scriptsize
\begin{tikzpicture}
\matrix (m) [matrix of math nodes,row sep=5.1em,column sep=1em,minimum width=1em]
{
	m=k\\
	{m= k-1}\\
	{m=  k-2}\\
	{m}\\
	{m=  0}\\
};
\path[-stealth]

;
\end{tikzpicture}
}
\;\;
\begin{tikzpicture}
\matrix (m) [matrix of math nodes,row sep=3.3em,column sep=0.5em,minimum width=1.5em]
{
	&&\ssign{k}&&\\
	&\ssign{k-1}&\dots &\ssign{k-1}&\\
	\ssign{k}&&\dots &&\ssign{k}\\
	&\dots &\dots &\dots &\\
	&&+&&\\
};
\path[-stealth]

(m-1-3)  edge node [above] {$$} (m-2-3)
(m-1-3)  edge node [above] {$$} (m-2-2)
(m-1-3)  edge node [above] {$$} (m-2-4)
(m-2-2)  edge node [above] {$$} (m-3-3)
(m-2-3)  edge node [above] {$$} (m-3-1)
(m-2-3)  edge node [above] {$$} (m-3-5)
(m-2-3)  edge node [above] {$$} (m-3-3)
(m-2-4)  edge node [above] {$$} (m-3-3)
(m-2-2)  edge node [above] {$$} (m-3-1)
(m-2-4)  edge node [above] {$$} (m-3-5)
(m-3-1)  edge node [above] {$$} (m-4-2)
(m-3-1)  edge node [above] {$$} (m-4-3)
(m-3-3)  edge node [above] {$$} (m-4-2)
(m-3-3)  edge node [above] {$$} (m-4-3)
(m-3-3)  edge node [above] {$$} (m-4-4)
(m-3-5)  edge node [above] {$$} (m-4-3)
(m-3-5)  edge node [above] {$$} (m-4-4)
(m-4-2)  edge node [above] {$$} (m-5-3)
(m-4-3)  edge node [above] {$$} (m-5-3)
(m-4-4)  edge node [above] {$$} (m-5-3)

(m-2-3)  edge [dashed, bend left=15] node [above] {$$} (m-1-3)
(m-2-2)  edge [dashed, bend left=15] node [above] {$$} (m-1-3)
(m-2-4)  edge [dashed, bend left=15] node [above] {$$} (m-1-3)
(m-3-3)  edge [dashed, bend left=15] node [above] {$$} (m-2-2)
(m-3-1)  edge [dashed, bend left=15] node [above] {$$} (m-2-3)
(m-3-5)  edge [dashed, bend left=15] node [above] {$$} (m-2-3)
(m-3-3)  edge [dashed, bend left=15] node [above] {$$} (m-2-3)
(m-3-3)  edge [dashed, bend left=15] node [above] {$$} (m-2-4)
(m-3-1)  edge [dashed, bend left=15] node [above] {$$} (m-2-2)
(m-3-5)  edge [dashed, bend left=15] node [above] {$$} (m-2-4)

(m-4-3)  edge [dashed, bend left=15] node [above] {$$} (m-3-5)
(m-4-4)  edge [dashed, bend left=15] node [above] {$$} (m-3-5)
(m-5-3)  edge [dashed, bend left=15] node [above] {$$} (m-4-2)
(m-5-3)  edge [dashed, bend left=15] node [above] {$$} (m-4-4)
(m-5-3)  edge [dashed, bend left=15] node [above] {$$} (m-4-3)
(m-4-3)  edge [dashed, bend left=15] node [above] {$$} (m-3-3)
(m-4-4)  edge [dashed, bend left=15] node [above] {$$} (m-3-3)
(m-4-2)  edge [dashed, bend left=15] node [above] {$$} (m-3-3)
(m-4-4)  edge [dashed, bend left=15] node [above] {$$} (m-3-3)
(m-4-2)  edge [dashed, bend left=15] node [above] {$$} (m-3-1)
(m-4-3)  edge [dashed, bend left=15] node [above] {$$} (m-3-1)
;
\end{tikzpicture}
\begin{tikzpicture}
\matrix (m) [matrix of math nodes,row sep=4.2em,column sep=1em,minimum width=1em]
{
	\\
	\\
	\bigoplus\\
	\\
	\\
};
\path[-stealth]
;
\end{tikzpicture}
\begin{tikzpicture}
\matrix (m) [matrix of math nodes,row sep=3.3em,column sep=0.5em,minimum width=1.5em]
{
	&&\ssign{k}&&\\
	&\ssign{k-1}&\dots &\ssign{k-1}&\\
	\ssign{k}&&\dots &&\ssign{k}\\
	&\dots &\dots &\dots &\\
	&&+&&\\
};
\path[-stealth]

(m-2-3)  edge [dashed] node [above] {$$} (m-1-3)
(m-2-2)  edge [dashed] node [above] {$$} (m-1-3)
(m-2-4)  edge [dashed] node [above] {$$} (m-1-3)
(m-3-3)  edge [dashed] node [above] {$$} (m-2-2)
(m-3-1)  edge [dashed] node [above] {$$} (m-2-3)
(m-3-5)  edge [dashed] node [above] {$$} (m-2-3)
(m-3-3)  edge [dashed] node [above] {$$} (m-2-3)
(m-3-3)  edge [dashed] node [above] {$$} (m-2-4)
(m-3-1)  edge [dashed] node [above] {$$} (m-2-2)
(m-3-5)  edge [dashed] node [above] {$$} (m-2-4)

(m-4-3)  edge [dashed] node [above] {$$} (m-3-5)
(m-4-4)  edge [dashed] node [above] {$$} (m-3-5)
(m-5-3)  edge [dashed] node [above] {$$} (m-4-2)
(m-5-3)  edge [dashed] node [above] {$$} (m-4-4)
(m-5-3)  edge [dashed] node [above] {$$} (m-4-3)
(m-4-3)  edge [dashed] node [above] {$$} (m-3-3)
(m-4-4)  edge [dashed] node [above] {$$} (m-3-3)
(m-4-2)  edge [dashed] node [above] {$$} (m-3-3)
(m-4-4)  edge [dashed] node [above] {$$} (m-3-3)
(m-4-2)  edge [dashed] node [above] {$$} (m-3-1)
(m-4-3)  edge [dashed] node [above] {$$} (m-3-1)
;
\end{tikzpicture}
\caption{
Action of $D(B_k)$ on $G(\mathcal{I})=\mathcal{A}_{(-)^k}\oplus \mathcal{C}_{+}$. 
The $m$th layer consists of  $2{k\choose m}$ simple  (one-dimensional) $B_k$-subquotients having the same sign (the signs $\pm$ correspond to the $g$-action):
$\binom{k}{m}$ of them  in the left summand are spanned by basis elements $(x_{i_1}\dots x_{i_m}g)^*$,
while  ${k\choose m}$ ones  in the right summand are spanned by  $(x_{i_1}\dots x_{i_m})^*$.
The  solid lines indicate the action of $x_i$'s and the dashed ones are for the $y_i$'s action.
}
\label{fig:arrows}
\end{figure}
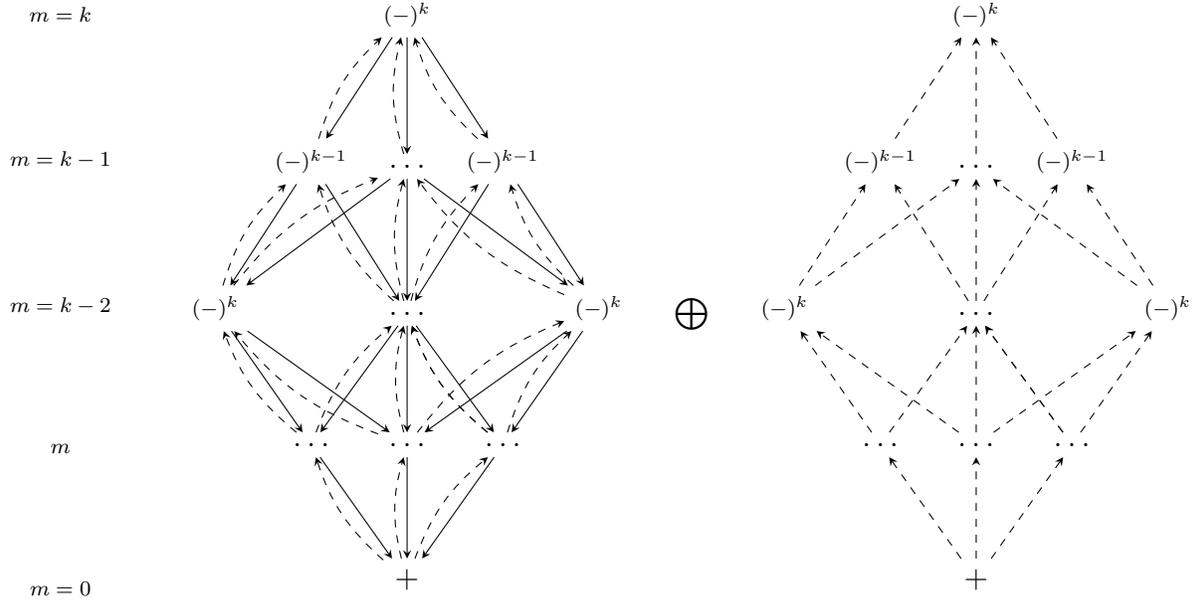

The analysis of the decomposition of $G(\mathcal{I}_{-})$ is completely analogous and we skip it.
Indeed, reproducing the above calculations in this case shows that $G(\mathcal{I}_{-})$ is isomorphic to $G(\mathcal{I})\otimes \mathcal{I}_{-}$.
\end{proof}

\begin{Cor}\label{Cor:Gprojectives}
  The modules $\Cat_+$ and $\Cat_-$ are $G$-projective.
\end{Cor}
\begin{proof}
$G(\mathcal{I}_{\pm})$ and their direct summands are $G$-projective by Lemma~\ref{lemma:directsummands}. Therefore due to Lemma~\ref{lemma:DecomCoad}, $\Cat_{\pm}$ are $G$-projective.
\end{proof}

\begin{Lemma}\label{lem:A-cpx}
Let $A$ be a $k$-algebra with an augmentation map $\epsilon\colon A\to k$, for a field $k$. Assume we have an exact complex  of $k$-vector spaces
\be\label{eq:resolution-k}
R\colon \; \dots \xrightarrow{f_{n+1}} k^{c_n}  \xrightarrow{\;f_{n}\;} k^{c_{n-1}}  \xrightarrow{f_{n-1}} \dots  \xrightarrow{\;f_{2}\;} k^{c_1}  \xrightarrow{\;f_{1}\;} k^{n_0} \to k\to 0 .
\ee
Let ${}_{\epsilon}R$ denotes the corresponding complex of $A$-modules with $k^{c_n}$ replaced by ${}_{\epsilon}k^{\oplus c_n}$.
Then for any $A$-module $M$ the complex $\Hom_A(M,{}_{\epsilon}R)$ is also exact.
\end{Lemma}
\begin{proof}
The cochain spaces of the complex $\Hom_A(M,{}_{\epsilon}R)$  are 
\be\label{eq:cochain-iso}
C^n:= \Hom_A(M,{}_{\epsilon}k^{\oplus c_n})\cong \Hom_A(M,{}_{\epsilon}k)\otimes k^{c_n} \ ,
\ee
and cochain maps $\widehat{f}_n\colon C^n \to C^{n-1}$   are given by $\widehat{f}_n\colon \phi \mapsto f_n\circ \phi$ for each $\phi\in C^n$. Using the isomorphism in~\eqref{eq:cochain-iso}, we can assume without loss of generality that $\phi = \psi\otimes v$ for some $\psi \in  \Hom_A(M,{}_{\epsilon}k)$ and  $v\in k^{c_n}$. On such vectors  $\widehat{f}_n(\phi) =  \psi \otimes  f_n(v)$, or $\widehat{f}_n = \id \otimes f_n$\footnote{For brevity, we omit the conjugation by the isomorphism from~\eqref{eq:cochain-iso}.}. 
Therefore, we have an isomorphism of complexes:
$$
\Hom_A(M,{}_{\epsilon}R) \cong  \Hom_A(M,{}_{\epsilon}k)\otimes R, 
$$
with cochain maps of the form $\id \otimes f_n$, and  exactness of $\Hom_A(M,{}_{\epsilon}R)$  follows from exactness of $R$.
\end{proof}

We can now construct the desired $G$-resolution:
\begin{Lemma}\label{lemma:G-resolution}
There is a $G$-resolution in $D(B_k)\mathrm{-mod}$ of the following form:
\begin{equation}\label{eq:resolution}
\dots \to \mathcal{C}_-^{\oplus a_3}\to \mathcal{C}_+^{\oplus a_2}\to \mathcal{C}_-^{\oplus a_1}\to \mathcal{C}_+^{\oplus a_0}\to \mathcal{I}\to 0
\end{equation}
with $a_n={k+n-1\choose n}$.
\end{Lemma}
\begin{proof}
We first construct an exact sequence of $D(B_k)$-modules  of the form~\eqref{eq:resolution} and then check that it is also $G$-exact.
Since the action of $x_i$ on $\Cat_{\pm}$ is trivial and $g$ acts as $h$, it suffices to construct an exact sequence in $B_k^*\mathrm{-mod}$. 

We have from~\eqref{eq:relationdD1} and~\eqref{eq:relationD2} that 
$$
 B_k^* = \langle y_1,\dots ,y_k,h\rangle \cong \Lambda \mathbb{C}^k\rtimes \C[\Z_2]
$$
where $\Lambda \mathbb{C}^k$ is the exterior algebra of $\mathbb{C}^k = \mathrm{Span}\{y_i, \, 1\leq i \leq k\}$ and the  isomorphism is obvious. Under the isomorphism, the  $B_k^*$-modules $\Cat_{\pm}$ are isomorphic to the vector space $\Lambda\mathbb{C}^k$ with $\Lambda\mathbb{C}^k$-action given by the multiplication $\wedge$ on $\Lambda\mathbb{C}^k$ and with the action $h.1=\pm 1$.

We recall  the `Koszul resolution' of the trivial module $\mathbb{C}$ over the exterior algebra\footnote{This is the complex dual to the one in~\cite[Ex.~17.21 (c)]{E} and composed with the augmentation map $\Lambda \C^k \to \C$, $y_i\to 0$, $1\to 1$.}:
	 \begin{equation}\label{eq:Koszul}
	 	\dots \to  S^2\left(\C^k\right) \otimes \Lambda\mathbb{C}^k\xrightarrow{\tilde{f}_2} S^1\left(\C^k\right)\otimes \Lambda\mathbb{C}^k\xrightarrow{\tilde{f}_1} S^0\left(\C^k\right)\otimes \Lambda\mathbb{C}^k\xrightarrow{\tilde{f}_0} \mathbb{C}\to 0,
	 \end{equation}
	 where  the subspaces $S^n\bigl(\C^k\bigr)$ of  the symmetric algebra  $S\bigl(\C^k\bigr)$ consist of elements of the form of $n$-fold tensor products, and $\tilde{f}_i$ are $\Lambda\mathbb{C}^k$-module maps  such that $\tilde{f}_{i+1} \circ \tilde{f}_i = 0$.

We are now able to construct a resolution of the form~\eqref{eq:resolution}. Note that the action of $h$ endows the cochain spaces in~\eqref{eq:resolution} with $\mathbb{Z}_2$ grading. Let $\Pi\colon\Cat_{\pm}\to\Cat_{\mp}$ denotes the corresponding parity shift operator, i.e.\ it is $\Lambda\mathbb{C}^k$-equivariant and sends $1$ to $1$. 
	  Then, the above Koszul complex~\eqref{eq:Koszul} can be extended to a $\mathbb{Z}_2$-equivariant one as follows:
	\begin{equation}\label{eq:resolution2}
	\dots \to  S^2\left(\C^k\right) \otimes \mathcal{C}_+\xrightarrow{f_2} S^1\left(\C^k\right)\otimes \mathcal{C}_-\xrightarrow{f_1} S^0\left(\C^k\right)\otimes \mathcal{C}_+\xrightarrow{f_0} \mathcal{I}\to 0
	\end{equation}
	  	where the tensor products are over $\mathbb{C}$ and we define
	  	 \begin{equation}\label{eq:fn}
	  	 f_n:=\left(\id_{S^{n}(\mathbb{C}^k)}\otimes\Pi\right)\circ\tilde{f}_n.
	  	 \end{equation}
The parity shift part in~\eqref{eq:fn} is necessary in order to make the cochain maps even, indeed the maps $\tilde{f}_n$ used in~\eqref{eq:Koszul} are odd. 
We note that the  complex~\eqref{eq:resolution2} is just a projective resolution of the trivial $B_k^*$-module\footnote{One can check that this is actually a minimal projective resolution.}.
	The formula for the multiplicities $a_n$ in~\eqref{eq:resolution} then follows from the fact that $\dim S^n(\mathbb{C}^k)={k+n-1\choose n}$.
	
	 In the remainder of the proof we show that the exact sequence in~\eqref{eq:resolution2} is in fact a $G$-resolution. 
	 All objects (except $\mathcal{I}$) are $G$-projective by Corollary~\ref{Cor:Gprojectives}.
	We can further check that the resolution is $G$-exact: If we apply the functor 
	\begin{equation}
	\Hom_{D(B_k)}\bigl(G(X),?\bigr)\cong\Hom_{B_k}\bigl(\U(X),\U(?)\bigr),
	\end{equation}
	with $X\in D(B_k)\mathrm{-mod}$ and via $\U$ forgetting the $B_k^*$ part of the $D(B_k)$-action, we obtain the complex
	\begin{align}\label{eq:resolutionHom}
	\dots
	\xrightarrow{\widehat{f}_{n+1}} \Hom_{B_k}\bigl(\U(X), \U(\mathcal{C}_{(-)^n})^{\oplus a_n}\bigr)\xrightarrow{\widehat{f}_n}
	\Hom_{B_k}\bigl(\U(X),\U(\mathcal{C}_{(-)^{n-1}})^{\oplus a_{n-1}}\bigr)\xrightarrow{\widehat{f}_{n-1}}\dots
	\end{align}
where $\U(\mathcal{C}_{\pm})$ are completely decomposed into copies of $I$ and $I_-$:
	\begin{equation}
	\U(\mathcal{C}_{\pm})\cong I^{\oplus 2^{k-1}}\oplus I_-^{\oplus 2^{k-1}}.
	\end{equation}
	We note that the maps $\widehat{f}_n$ in~\eqref{eq:resolutionHom} are given by post-composing with the maps from~\eqref{eq:resolution2}: $\widehat{f}_n\colon \phi \mapsto f_n\circ \phi$.
	Therefore, as the cochain maps from (\ref{eq:resolution2}) preserve the $h$-action, the complex~\eqref{eq:resolutionHom} decomposes into a direct sum of complexes, one with cochain spaces  $C^n=\Hom_{B_k}(\U(X), I^{\oplus a_n 2^{k-1}} )$ and the other with $C^n=\Hom_{B_k}(\U(X), I_-^{\oplus a_n 2^{k-1}} )$. It is therefore  enough to show exactness for each copy separately.  Recall that the forgetful functor $\U$ is exact and therefore its image on~\eqref{eq:resolution2} is split on direct sum of two resolutions, one with direct sums of $I$ and the other with $I_-$.
 They are both resolutions in vector spaces after applying the fiber functor. 
 Then applying Lemma~\ref{lem:A-cpx} for each of these resolutions and the case $A=B_k$ and $M=\U(X)$  proves  exactness of~\eqref{eq:resolutionHom}.
\end{proof}

Finally, we can apply the general theory of comonad cohomology to prove the formula in Theorem~\ref{proposition:DYSweedler} for $\dim H^n_{DY}(B_k\mathrm{-mod})$.

\begin{proof}[Proof of Theorem~\ref{proposition:DYSweedler} for the identity functor]
By Theorem~\ref{theorem:Main}, we can reformulate Da\-vy\-dov-Yet\-ter cohomology of the identity functor as the comonad cohomology of the co\-mo\-nad~$G$ for the case when the coefficients $\mathsf{X}=\mathsf{Y}=\mathcal{I}$ are the trivial $D(B_k)$-module.
 Theorem~\ref{theorem:Comparison} allows to use any $G$-resolution to compute the cohomologies. 
 We compute the comonad cohomology (and hence DY cohomology) by applying the respective coefficient functor $\Hom_{D(B_k)}(-,\mathcal{I})$ to the $G$-resolution constructed in Lemma~\ref{lemma:G-resolution}. The statement for the identity functor in (\ref{eq:SweedlerDY}) follows immediately from observing that $\Hom_{D(B_k)}(\mathcal{C}_-,\mathcal{I})=0$ and $ \Hom_{D(B_k)}(\mathcal{C}_+,\mathcal{I})=\C$. 
 \end{proof}
 
 \begin{Remark}\label{rem:DY-exmp}
	We can write down  generators of $H^2_{DY}(B_k\mathrm{-mod})$ explicitly.
For a Hopf algebra $H$, the algebra of natural transformations $\Nat_{\Hmod}(\otimes^n,\otimes^n)$ is isomorphic to the subalgebra in $H^{\otimes n}$ that commutes with the $n$-fold coproduct $\Delta^{(n)}(h)$ for any $h\in H$, as was observed
 in~\cite[Sec.~6]{ENO}.
	 In the following table, $f=\sum_i f^{i}_{1}\otimes\dots \otimes f^{i}_{n}\in H^{\otimes n}$ corresponds to the natural transformation defined by 
	  $\eta_f(v_1\otimes\dots\otimes v_n):= \sum_i f^{i}_{1}.v_1\otimes\dots\otimes f^{i}_{n}.v_n$, for $v_k\in V_k$. 
	 \begin{center}
	 {\footnotesize
		\begin{tabular}{|c|c|c|c|c|c|c|c|c|c|}\hline
			Hopf algebra & Generators of
 $H^2_{DY}(B_k\mathrm{-mod})$\\ \hline
			$B_1$ & $x\otimes xg$\\ \hline
			$B_2$& $x_1\otimes x_1g,\;\; x_2\otimes x_2g,\;\; x_1\otimes x_2g+x_2\otimes x_1g$\\ \hline
			$B_3$& $x_1\otimes x_1g,\; x_2\otimes x_2g,\; x_3\otimes x_3g$,\\ 
			&$ x_1\otimes x_3g+x_3\otimes x_1g,\; x_2\otimes x_1g+x_1\otimes x_2g,
			\; x_2\otimes x_3g+x_3\otimes x_2g$\\\hline
		\end{tabular}
		}
	\end{center}
These natural transformations 
define infinitesimal deformations of the monoidal structure of the identity functor.
Due to the fact that $H^3_{DY}=0$ in this case, the  deformations have no obstructions.
\end{Remark}

In the remainder of this section we prove the formula in Theorem~\ref{proposition:DYSweedler} for the forgetful functor, i.e.\ for $\dim H^n_{DY}(\mathcal{U}_{B_k})$.
Using  Theorem~\ref{cor:HochschildasCoefficients} and the discussion after it, we compute the DY cohomology of the forgetful functor via the DY cohomology of the identity functor with second coefficient $\left(B^*_{k,\mathrm{coreg}},\beta_c\right)$, recall its definition  in~\eqref{eq:beta}. In the following lemma we decompose the $D(B_k)$-module corresponding to $\left(B^*_{k,\mathrm{coreg}},\beta_c\right)$. 

\begin{Lemma}\label{lemma:DecomCoreg}
	We have a decomposition of the $D(B_k)$-module $\left(B^*_{k,\mathrm{coreg}},\beta_c\right)$:
	\begin{equation}
	\left(B_{k,\mathrm{coreg}}^*,\beta_c\right)\cong\mathcal{A}_{(-)^{k+1}}\oplus \mathcal{B}_{(-)^k}.
	\end{equation}
\end{Lemma}
\begin{proof}
We first analyze the action of the subalgebra $B_k\subset D(B_k)$.
	The action of $B_k$ on $\left(B^*_{k,\mathrm{coreg}},\beta_c\right)$ is just the coregular action.   
	We have the following actions of $x_j$ and $g$:
	\begin{align}
	g.(x_{i_1}\dots x_{i_m})^*=& (x_{i_1}\dots x_{i_m}g)^*,\nonumber\\
	g.(x_{i_1}\dots x_{i_m}g)^*=& (x_{i_1}\dots x_{i_m})^*,\nonumber\\
	x_j.(x_{i_1}\dots x_{i_m})^*= &\begin{cases}
	(-1)^{m-l}(x_1\dots\hat{x}_{i_l}\dots x_{i_m})^*&\textrm{for}\;\; i_l=j\\
	0& \text{for}\;\;i_l\neq j, \; 1\leq l\leq m,
	\end{cases}\nonumber\\
	x_j.(x_{i_1}\dots x_{i_m}g)^*= &\begin{cases}
	(-1)^{m-l+1}(x_1\dots \hat{x}_{i_l}\dots x_{i_m}g)^*&\textrm{for}\;\; i_l=j\\
	0& \text{for}\;\;i_l\neq j, \; 1\leq l\leq m.
	\end{cases}
	\end{align}
It is clear that this is a free action and  isomorphic to the regular $B_k$-module.
Therefore, 
$B^*_{k,\mathrm{coreg}}$ as a $B_k$-module can be decomposed as
	\begin{equation}
	B^*_{k,\mathrm{coreg}}=P_{+}\oplus P_-,
	\end{equation}
	where in a basis we have the identification
	\begin{equation}
	P_+\cong B_k.\bigl((x_1\dots x_k)^*+(x_1\dots x_kg)^*\bigr)\quad  , \quad  P_- \cong B_k.\bigl((x_1\dots x_k)^*-(x_1\dots x_kg)^*\bigr).
	\end{equation}
	
The action of the subalgebra $B_k^*$ is given by $\beta_c\colon B_k^*\otimes B_k^*\to B_k^*$, recall~\eqref{eq:beta}.
Using the formula~\eqref{eq:coproductformula} twice, we get
	\begin{align}
	y_{i_l}.(x_{i_1}\dots\hat{x}_{i_l}\dots x_{i_m})^*=&(-1)^{m-l}(x_{i_1}\dots x_{i_m})^*+(-1)^{l-1}(x_{i_1}\dots x_{i_m}g)^*,\nonumber\\
	y_{i_l}.(x_{i_1}\dots\hat{x}_{i_l}\dots x_{i_m}g)^*=&(-1)^{l}(x_{i_1}\dots x_{i_m})^*+(-1)^{m-l-1}(x_{i_1}\dots x_{i_m}g)^*.
	\end{align}
	In particular, we obtain on the basis elements of the $B_k$-submodules $P_+, P_-$:
	\begin{align}
	y_{i_l}&.((x_{i_1}\dots\hat{x}_{i_l}\dots x_{i_m})^*\pm(x_{i_1}\dots\hat{x}_{i_l}\dots x_{i_m}g)^*)\\
	&=\begin{cases}
	0&\textrm{for}\;\;(-1)^m=\mp\\
	2(-1)^l(\pm 1)((x_{i_1}\dots x_{i_m})^*\mp(x_{i_1}\dots x_{i_m}g)^*)& \text{for}\;\;(-1)^m=\pm.
	\end{cases}
	\end{align}
	Hence, we can identify the cyclic vector $v_{(-)^{k+1}}:=(x_{1}\dots x_{k})^*+(-1)^{k+1}(x_{1}\dots x_{k}g)^*$ such that the $B_k$ submodule $P_{(-)^{k+1}}\cong B_k.v_{(-)^{k+1}}$ becomes the module $\mathcal{A}_{(-)^{k+1}}$ under the action of $B_k^*\subset D(B_k)$.
	This follows again from the fact that this module is indecomposable and admits the action $g=-h$, which implies that it contains $\mathcal{A}_{(-)^{k+1}}$ as a simple submodule due to Lemma~\ref{lem:D-simples}.
	 Simlarly $P_{(-)^k}$ becomes $\mathcal{B}_{(-)^k}$ under the action of $B_k^*$.
\end{proof}
The comonad cohomology with the coefficient functor $\Hom_{\mathcal{Z}(B_k\mathrm{-mod})}(?,\mathsf{Y})$ preserves direct sums. Thus, we can simply neglect the summand $\mathcal{A}_{\pm}$ in $\mathsf{Y}=\left(B^*_{k,\mathrm{coreg}},\beta_c\right)$ because it is injective and makes the functor $\Hom_{\mathcal{Z}(B_k\mathrm{-mod})}(?,\mathcal{A}_{\pm})$ exact.
\begin{Cor}
	We have
	\begin{equation}
	H^{n}\left(?,\Hom_{D(B_k)}(!,\mathcal{A}_{\pm}\oplus \mathcal{B}_{\mp})\right)_G\cong H^{n}\left(?,\Hom_{D(B_k)}(!,\mathcal{B}_{\mp})\right)_G
	\end{equation}
	for all $n>0$.
\end{Cor}
\begin{proof}[Proof of Theorem~\ref{proposition:DYSweedler} for the forgetful functor]
	The statement for the forgetful functor follows from the identities $\Hom_{D(B_k)}(\mathcal{C}_-,\mathcal{B}_-)=0$ and $\Hom_{D(B_k)}(\mathcal{C}_+,\mathcal{B}_-)\cong\C$ for odd $k$ and $\Hom_{D(B_k)}(\mathcal{C}_-,\mathcal{B}_+)=0$ and $\Hom_{D(B_k)}(\mathcal{C}_+,\mathcal{B}_+)\cong\C$ for even $k$.
\end{proof}
\begin{Remark}
The formula~\eqref{eq:SweedlerDY} for the Davydov-Yetter cohomology of the forgetful functor can be obtained by using the following well known isomorphism for a Hopf algebra $H$ and an $H$-bimodule $M$:
\begin{equation}
\mathrm{HH}^{\bullet}(H,M)\cong\mathrm{Ext}^{\bullet}_{H\otimes H^{op}}(H,M)\cong\mathrm{Ext}^{\bullet}_H(k,M_{ad}),
\end{equation}
where $k$ is the trivial module and $M_{ad}$ is the adjoint representation corresponding to the bimodule $M$. Specifically, for the trivial bimodule $M=k$, the latter is just $\mathrm{Ext}_H^{\bullet}(k,k)$.
The Davydov-Yetter cohomology of the forgetful functor is isomorphic to the Hochschild cohomology of the dual Hopf algebra for $M=k$. 
It is thus enough to compute $\mathrm{Ext}_{B^*_k}^n(I,I)$. 
 This can be done with standard homological algebra techniques. In fact, the minimal projective resolution of the trivial module $I$ is identical to the one in~\eqref{eq:resolution} restricted to the subalgebra $B^*_k$. 
 The calculation is therefore analogous to the end of the proof of Theorem~\ref{proposition:DYSweedler} for the identity functor.
\end{Remark}

\end{document}